\documentclass[11pt]{amsart}
\usepackage[utf8]{inputenc}
\usepackage{amsmath,amssymb, mathrsfs}
\usepackage{bbm}
\usepackage{lipsum}

\usepackage{centernot}%negare simboli meglio
\usepackage{cite}
\usepackage{xcolor}

\usepackage{tikz}
\usetikzlibrary{arrows}
\usetikzlibrary{patterns}
\usetikzlibrary{shapes}
\usepackage{mathrsfs}
\usepackage{bbm}
\usepackage{enumerate}
\usepackage{tikz}
\usepackage{tikzcd}

\usepackage{amssymb}
\usepackage{tikz-cd}
\usetikzlibrary{cd}

\newcommand{\R}{\mathbb{R}}
\newcommand{\N}{\mathbb{N}}

\newcommand{\de}{\partial}
\renewcommand{\-}{\backslash}
\newcommand{\D}{\mathbb{D}}
\newcommand{\B}{\mathbb{B}}

\renewcommand{\a}{\alpha}

\newcommand{\f}{\varphi}
\newcommand{\e}{\varepsilon}

\newcommand{\transv}{\mathrel{\text{\tpitchfork}}}
\makeatletter
\newcommand{\tpitchfork}{%
  \raise-0.1ex\vbox{
    \baselineskip\z@skip
    \lineskip-.52ex
    \lineskiplimit\maxdimen
    \m@th
    \ialign{##\crcr\hidewidth\smash{$-$}\hidewidth\crcr$\pitchfork$\crcr}
  }%
}
\makeatother

\newcommand{\vol}{\mathrm{vol}}

\newcommand{\RP}{\mathbb{R}\mathrm{P}}
\newcommand{\PP}{\mathbb{P}}

\renewcommand{\P}{\mathbb{P}}
\newcommand{\E}{\mathbb{E}}
\newcommand{\nrw}{\Rightarrow}
\newcommand{\spt}{\text{supp}}
\newcommand{\g}[3]{\mathcal{G}^{#1}(#2,\R^{#3})}

\newcommand{\Cr}[3]{\mathcal{C}^{#1}(#2,\R^{#3})}

\usepackage{hyperref}
\hypersetup{
	colorlinks=true,       % false: boxed links; true: colored links
	linkcolor=blue,         
	citecolor=blue}

 \usepackage[usenames,dvipsnames]{pstricks}
 \usepackage{epsfig}
 \usepackage{pst-grad} % For gradients
 \usepackage{pst-plot} % For axes

\newtheorem{thm}{Theorem}
\newtheorem{lemma}[thm]{Lemma}
\newtheorem{cor}[thm]{Corollary}
\newtheorem{prop}[thm]{Proposition}
\newtheorem{claim}[thm]{Claim}

\theoremstyle{definition}

\newtheorem{defi}[thm]{Definition}
\newtheorem{remark}[thm]{Remark}

\newtheorem{example}[thm]{Example}
\newcommand{\be}{\begin{equation}}
\newcommand{\ee}{\end{equation}}

\usepackage{mathtools} %scommentare per far numerare solo le equazioni che sono chiamate nel testo
\mathtoolsset{showonlyrefs,showmanualtags} %scommentare per far numerare solo le equazioni che sono chiamate nel testo

\numberwithin{equation}{section}

\title[Singularities of polynomial maps]{Maximal and typical topology of real polynomial singularities}
\author{Antonio Lerario and Michele Stecconi}

\begin{document}
\begin{abstract}Given a semialgebraic set $W\subseteq J^{r}(S^m, \R^k)$ and a polynomial map $\psi:S^m\to \R^k$ with components of degree $d$, we investigate the structure of the semialgebraic set $j^r\psi^{-1}(W)\subseteq S^m$ (we call such a set a ``singularity''). 

Concerning the upper estimate on the topological complexity of a polynomial singularity, we sharpen the classical bound $b(j^r\psi^{-1}(W))\leq O(d^{m+1})$, proved by Milnor \cite{Milnor}, with
\be\label{eq:abstract} b(j^r\psi^{-1}(W))\leq O(d^{m}),\ee
which holds for the generic polynomial map. 

For what concerns the ``lower bound'' on the topology of $j^r\psi^{-1}(W)$, we prove a general semicontinuity result for the Betti numbers of the zero set of $\mathcal{C}^0$ perturbations of smooth maps -- the case of $\mathcal{C}^1$ perturbations is the content of Thom's Isotopy Lemma (essentially the Implicit Function Theorem). This result is of independent interest and it is stated for general maps (not just polynomial); this result implies that small continuous perturbations of $\mathcal{C}^1$ manifolds have a richer topology than the one of the original manifold.

Keeping \eqref{eq:abstract} in mind, we compare the extremal case with a random one and prove that on average the topology of $j^r\psi^{-1}(W)$ behaves as the ``square root'' of its upper bound: for a random Kostlan map $\psi:S^m\to \R^k$ with components of degree $d$ and $W\subset J^{r}(S^m, \R^k)$ semialgebraic, we have:
\be \mathbb{E}b(j^{r}\psi^{-1}(W))=\Theta(d^{\frac{m}{2}}).\ee
This generalizes classical results of Edelman-Kostlan-Shub-Smale from the zero set of a random map, to the structure of its singularities.
%Using the tools that we have developed in \cite{DTGRF1}, we study properties of random Kostlan polynomial maps (viewed as random variables in the space of $C^{\infty}$-maps, see Theorem \ref{thm:Kostlan}). We apply these tools to the study of problems in random real algebraic geometry, with particular emphasis on the local structure of ``random singularities'' (i.e. the set of points where a map has some high-order jet of a prescribed type).
%This study leads to a generalized ``square-root law'' for the topology (Betti numbers or number of points) of a random singularity (Theorem \ref{thm:bettiorder} and Theorem \ref{thm:sqrlaw}): as the degree goes to infinity, the expected value of this number grows like the square root of the corresponding deterministic upper bound (most of the times coming from complex algebraic geometry).
%Finally, we establish two technical results of independent interest (used for the deterministic estimate of the topology of jet-type singularities and for the lower bound on its expectation): first we obtain Morse inequalities for stratified spaces that are ``almost'' semialgebraic (Theorem \ref{thm:bettibound} and Theorem \ref{thm:strat}), second we prove a semicontinuity result for the topology of the zero set of a nondegenerate equation under a small $\mathcal{C}^0$ perturbation of this equation (Theorem \ref{thm:semiconttop}).
\end{abstract}
\maketitle

\section{Introduction}
In this paper we deal with the problem of understanding the structure of the singularities of polynomial maps
\be \psi:S^m\to \R^k,\ee
where each component of $\psi=(\psi_1, \ldots, \psi_k)$ is the restriction to the sphere of a homogeneous polynomial of degree $d$. For us ``singularity'' means the set of points in the sphere where the $r$-jet extension $j^r\psi:S^n\to J^{r}(S^n, \R^k)$ meets a given semialgebraic set $W\subseteq J^{r}(S^n, \R^k).$ Example of these type of singularities are: zero sets of polynomial functions, critical points of a given Morse index of a real valued function or the set of Whitney cusps of a planar map.

Because we are looking at \emph{polynomial} maps, this problem has two different quantitative faces, which we both investigate in this paper. 

(1) From one hand we are interested in understanding the \emph{extremal} cases, meaning that, for fixed $m, d$ and $k$ we would like to know how complicated can the singularity be, at least in the generic case. 

(2) On the other hand, we can ask what is the \emph{typical} complexity of such a singularity. Here we adopt a measure-theoretic point of view and endow the space of polynomial maps with a natural Gaussian probability measure, for which it makes sense to ask about expected properties of these singularities, such as their Betti numbers.

\subsection{Quantitative bounds, the h-principle and the topology semicontinuity}
Measuring the complexity of  $Z=j^r\psi^{-1}(W)$ with the sum $b(Z)$ of its Betti numbers, problem (1) above means producing a-priori upper bounds for $b(Z)$ (as a function of $m, d, k$) as well as trying to realize given subsets of the sphere as $j^r\psi^{-1}(W)$ for some $W$ and some map $\psi$. 

For the case of the zero set $Z=\psi^{-1}(0)$ of a polynomial function $\psi:S^m\to\R$ of degree $d$, the first problem is answered by a Milnor's type bound\footnote{Milnor's bound \cite{Milnor} would give $b(Z)\leq O(d^{m+1})$, whereas \cite[Proposition 14]{LerarioJEMS} gives the improvement $b(Z)\leq O(d^m)$. In the context of this paper the difference between these two bounds is relevant, especially because when switching to the probabilistic setting it will give the so called ``generalized square root law''.} $b(Z)\leq O(d^m)$ and the second problem by Seifert's theorem: every smooth hypersurface in the sphere can be realized (up to ambient diffeomorphisms) as the zero set of a polynomial function. 

In the case of more general singularities, both problems are more subtle. The problem of giving a good upper bound on the complexity of $Z=j^r\psi^{-1}(W)$ will require us to develop a quantitative version of stratified Morse Theory for semialgebraic maps (Theorem \ref{thm:strat}). We use the word ``good'' because there is a vast literature on the subject of quantitative semialgebraic geometry, and it is not difficult to produce a bound of the form $b(Z)\leq O(d^{m+1})$; instead here (Theorem \ref{thm:bound} and Theorem \ref{thm:bound2}) we prove the following result.

\begin{thm}\label{thm:bettibound}For the generic polynomial map $\psi:S^m\to \R^k$ with components of degree $d$, and for $W\subseteq J^{r}(S^m, \R^k)$ semialgebraic, we have: 
\be \label{eq:estib}b(j^{r}\psi^{-1}(W))\leq O(d^m).\ee
(The implied constant depends on $W$.)
\end{thm}

In the case $W$ is algebraic we do not need the genericity assumption on $\psi$ for proving \eqref{eq:estib}, but in the general semialgebraic case some additional complications arise and this assumption allows to avoid them through the use of  Theorem \ref{thm:strat}. We believe, however, that \eqref{eq:estib} is still true even in the general case\footnote{In the algebraic case in fact one can use directly Thom-Milnor bound, but in the general semialgebraic case it is necessary first to ``regularize'' the semialgebraic set, keeping control on its Betti numbers. In the algebraic (or even the basic semialgebraic case) this is the procedure of Milnor \cite{Milnor}, in the general semialgebraic case it is not clear what this controlled regularization procedure would be. The nondegeneracy assumption on the jet allows us to avoid this step.}. Moreover, for our scopes the genericity assumption is not restrictive, as it fits in the probabilistic point of view of the second part of the paper, where a generic property is a property holding with probability one.

For what concerns the realizability problem, as simple as it might seem at first glance, given $W\subseteq J^{r}(S^m, \R^k)$ it is not even trivial to find a map $f:S^m\to \R^k$ whose jet is transversal to $W$ and such that $b(j^rf^{-1}(W))>0$ (we prove this in Corollary \ref{cor:topositive}). 

Let us try to explain carefully what is the subtlety here. In order to produce such a map, one can certainly produce a section of the jet bundle $\sigma:S^m\to J^{r}(S^m, \R^k)$ which is transversal to $W$ and such that $b(\sigma^{-1}(W))>0$ (this is easy). However, unless $r=0$, this section needs not to be holonomic, i.e. there might not exist a function $f:S^m\to \R^k$ such that $\sigma=j^{r}f$.  

We fix this first issue using an h-principle argument: the Holonomic Approximation Theorem \cite[p. 22]{eliash} guarantees that, after a small $\mathcal{C}^0$ perturbation of the whole picture, we can assume that there is a map $f:S^m\to \R^k$ whose jet $j^{r}f$ is $\mathcal{C}^0$ close to $\sigma$. 

There is however a second issue that one needs to address. In fact, if the jet perturbation was $\mathcal{C}^1$ small (i.e. if $\sigma$ and $j^rf$ were $\mathcal{C}^1$ close), Thom's Isotopy Lemma would guarantee that $\sigma^{-1}(W)\sim j^rf^{-1}(W)$ (i.e. the two sets are ambient diffeomorphic), but the perturbation that we get from the Holonomic Approximation Theorem is guaranteed to be only $\mathcal{C}^0$ small! To avoid this problem we prove the following general result on the semicontinuity of the topology of small $\mathcal{C}^0$ perturbations (see Theorem \ref{thm:semiconttop} below for a more precise statement).

\begin{thm}\label{thm:nonho}Let $S, J$ be smooth manifolds, $W\subseteq J$ be a closed cooriented submanifold %\footnote{In fact here we can take $W$ to be stratified in the sense of \cite[Exercise 15, Chapter 3.2]{Hirsch}.}
 and $\sigma\in \mathcal{C}^{1}(S, J)$ such that $\sigma\pitchfork W$. Then for every $\gamma\in  \mathcal{C}^{1}(S, J) $ which is sufficiently close to $\sigma$ in the $\mathcal{C}^0$-topology and such that $\gamma\pitchfork W$, we have:
\be b(\gamma^{-1}(W))\geq b(\sigma^{-1}(W)).\ee
\end{thm}
In particular we see that if small $\mathcal{C}^1$ perturbations of a regular equation preserve the topology of the zero set, still if we take just small $\mathcal{C}^0$ perturbations the topology of such zero set can only increase.

To apply Theorem \ref{thm:nonho} to our original question we consider $S=S^m$ and $J=J^{r}(S^m, \R^k)$, $W\subseteq J^{r}(S^m, \R^m)$ is the  semialgebraic set defining the singularity and $\sigma:S^m\to J^r(S^m, \R^k)$ is the (possibly non-holonomic) section such that $\sigma\pitchfork W$ and  $b(\sigma^{-1}(W))>0$. Moreover we can construct $\sigma$ in such a way that its image meets only a small (relatively compact and cooriented) subset of the smooth locus of $W$. Then for every $f\in  \mathcal{C}^{r+1}(S^m, \R^k) $ with $\tau=j^rf$ sufficiently close to $\sigma$ and such that $j^rf\pitchfork W$, we have:
\be \label{eq:app}b(j^rf^{-1}(W))\geq b(\sigma^{-1}(W))>0.\ee
(We will use the content of Corollary \ref{cor:topositive} and the existence of a function $f$ such that \eqref{eq:app} holds in the second part of the paper for proving the convergence of the expected Betti numbers of a random singularity.)

\subsection{The random point of view and the generalized square-root law}Switching to the random point of view offers a new perspective on these problems: from Theorem \ref{thm:bettibound} we have an extremal bound \eqref{eq:estib} for the complexity of polynomial singularities, but it is natural to ask how far is this bound from the typical situation. 
Of course, in order to start talking about randomness, we need to choose a probability distribution on the space of (homogeneous) polynomials. It is natural to require that this distribution is gaussian, centered, and that it is invariant under orthogonal changes of variables (in this way there are no preferred points or directions in the sphere). If we further assume that the monomials are independent, this distribution is unique (up to multiples), and called the \emph{Kostlan distribution}. 

To be more precise, this probability distribution is the measure on $\R[x_0,\dots,x_{m}]_{(d)}$ (the space of homogeneous polynomials of degree $d$) induced by the gaussian random polynomial:
\be \label{eq:kosdef}
P(x)=\sum_{ |\a|=d}\xi_\a\cdot \left(\frac{d!}{\alpha_0!\cdots\alpha_m!}\right)^{1/2}  x_0^{\alpha_0}\cdots x_m^{\alpha_m},
\ee
where $\{\xi_\a\}$ is a family of standard independent gaussian variables. A list of $k$ independent Kostlan polynomials $P=(P_1, \ldots, P_k)$ defines a random polynomial map:
\be \psi=P|_{S^m}\to \R^k.\ee
In particular, it is now natural to view such a $\psi$ as a random variable in the space $\mathcal{C}^{\infty}(S^m, \R^k)$ and to study the differential topology of this map, such as the behavior of its singularities, described a preimages of jet submanifolds $W\subseteq J^{r}(S^m, \R^k)$ in the previous section.

In this direction, it has already been observed by several authors, in different contexts, that random real algebraic geometry seems to behave as the ``square root'' of generic complex geometry. Edelman and Kostlan \cite{Ko90, EdelmanKostlan95} were the first to observe this phenomenon: a random Kostlan polynomial of degree $d$ in one variable has $\sqrt{d}$ many real zeroes, on average\footnote{In the notation of the current paper this correspond to the case of $\psi:S^1\to \R$ of degree $d$, whose expected number of zeroes is $2\sqrt{d}$. The multiplicative constant ``$2$'' appears when passing from the projective to the spherical picture}. Shub and Smale \cite{shsm} generalized this result and proved that the expected number of zeroes of a system of $m$ Kostlan equations of degrees $(d_1, \ldots, d_m)$ in $m$ variables is $\sqrt{d_1\cdots d_m}$ (the bound coming from complex algebraic geometry would be $d_1\cdots d_m$).

Moving a bit closer to topology, B\"urgisser \cite{buerg:07} and Podkorytov \cite{Podkorytov} proved that the expectation of the Euler characteristic of a random Kostlan algebraic set has the same order of the square-root of the Euler characteristic of its complex part (when the dimension is even, otherwise it is zero).
A similar result for the Betti numbers has also been proved by Gayet and Welschinger \cite{GaWe1, GaWe2, GaWe3}, and by Fyodorov, Lerario and Lundberg \cite{FLL} for invariant distributions. 

Using the language of the current paper, these results correspond to the case of  a polynomial map $\psi:S^m\to \R^k$ and to the ``singularity'' $Z=j^0\psi^{-1}(W)$, where 
\be W=S^m\times \{0\}\subset J^{0}(S^m, \R^k)=S^m\times \R^k\ee
 and $j^0\psi(x)=(x, \psi(x))$ is the section given by the map $\psi$ itself. Here we generalize these results and prove that a similar phenomenon is a very general fact of Kostlan polynomial maps.
 
 \begin{thm}\label{thm:bettiorderintro}
Let $W\subset J^r(S^m,\R^k)$ be a closed intrinsic\footnote{We say that $W\subset J^r(S^m,\R^k)$ is intrinsic if it is invariant under diffeomorphisms of $S^m$, see Definition \ref{def:intrinsic}. This property it is satisfied in all natural examples.} semialgebraic set of positive codimension. If $\psi:S^m\to \R^k$ is a random Kostlan polynomial map, then
\be \label{eq:ineqbetti2}
\E b(j^r\psi^{-1}(W))=\Theta(d^{\frac{m}{2}}).\footnote{We write $f(d)=\Theta(g(d))$ if there exist constants $a_1, a_2>0$ such that $a_1 g(d)\leq f(d)\leq a_2f(d)$ for all $d\geq d_0$ sufficiently large.}
\ee
(The implied constants depend on $W$.)

\end{thm}

We call the previous Theorem \ref{thm:bettiorderintro} the ``generalized square root law'' after comparing it with the extremal inequality $b(j^{r}\psi^{-1}(W))\leq O(d^m)$
 from Theorem \ref{thm:bettibound}, whose proof is ultimately based on bounds coming from complex algebraic geometry\footnote{The reader can now appreciate the estimate $O(d^m)$ instead of $O(d^{m+1})$ from Theorem \ref{thm:bettibound}.}. In the case $W$ has codimension $m$ (i.e. when we expect $j^r\psi^{-1}(W)$ to consist of points), we actually sharpen \eqref{eq:ineqbetti2} and get the explicit asymptotic to the leading order, see Theorem \ref{thm:sqrlaw} below. Moreover, a similar result holds for every fixed Betti number $b_i(j^r\psi^{-1}(W))$ when $i$ is in the range $0\leq i\leq m-\mathrm{codim}(W),$ see Theorem \ref{thm:bettiorder} for a detailed statement.
  
\begin{remark}The ingredients for the proof of Theorem \ref{thm:bettiorderintro} are: Theorem \ref{thm:strat} for the upper bound and Corollary \ref{cor:topositive} for the lower bound. The main property that we use in this context is the fact that a Kostlan map $\psi:S^m\to \R^k$ has a rescaling limit when restricted to a small disk $D_d=D(x, d^{-1/2})$ around any point $x\in S^m$. In other words, one can fix a diffeomorphism $a_d:\mathbb{D}^m\to D_d$ of the standard disk $\mathbb{D}^m$ with the small spherical disk $D(x, d^{-1/2})\subset S^m$ and see that the sequence of random functions:
\be X_d=\psi\circ a_d:\mathbb{D}^m\to \R^k\ee
converges to the Bargmann-Fock field, see Theorem \ref{thm:Kostlan}. 
In a recent paper \cite{DTGRF1} we introduced a general framework for dealing with random variables in the space of smooth functions and their differential topology -- again we can think of $X_d\in C^{\infty}(\mathbb{D}^m, \R^k)$ as a sequence of random variables of this type. The results from \cite{DTGRF1}, applied to the setting of random Kostlan polynomial maps are collected in Theorem \ref{thm:Kostlan} below, which lists the main properties of the rescaled Kostlan polynomial $X_d$. Some of these properties are well-known to experts working on random fields, but some of them seem to have been missed. Moreover, we believe that our language is more flexible and well-suited to the setting of differential topology, whereas classical references look at these random variables from the point of view of functional analysis and stochastic calculus.

Of special interest from Theorem \ref{thm:Kostlan} are properties (2), (5) and (7), which are closely related. In fact (2) and (5) combined together tells that open sets $U\subset \mathcal{C}^{\infty}(\mathbb{D}^m, \R^k)$ which are defined by open conditions on the $r$-jet of $X_d$, have a positive limit probability when $d\to \infty$. Property (7), tells that the law for Betti numbers of a random singularity $Z_d=j^rX_d^{-1}(W)$ has a limit. (Even in the case of zero sets this property was not noticed before, see Example \ref{ex:welsch}.)

We consider Theorem \ref{thm:Kostlan} as a practical tool that people interested in random algebraic geometry can directly use, and we will show how to concretely use this tool in a list of examples that we give in Appendix \ref{sec:examples}.
\end{remark}

\begin{remark}The current paper, and in particular the generalized square-root law Theorem \ref{thm:bettiorderintro}, complement recent work of Diatta and Lerario \cite{DiattaLerario} and Breiding, Keneshlou and Lerario \cite{BKL}, where tail estimates on the probabilities of the maximal configurations are proved. 
\end{remark}
\subsection{Structure of the paper}In Section \ref{sec:strat} we prove a quantitative semialgebraic version of stratified Morse Theory, which is a technical tool needed in the sequel, and in Section \ref{sec:quantitative} we prove Theorem \ref{thm:bound} and Theorem \ref{thm:bound2} (whose combination give Theorem \ref{thm:bettibound}).  In section \ref{sec:holo} we discuss the semicontinuity of topology under holonomic approximation and prove Theorem \ref{thm:semiconttop} (which is Theorem \ref{thm:nonho} from the Introduction). In Section \ref{sec:RAG} we introduce the random point of view and prove the generalized square-root law. Appendix 1 contains three short examples of use the random techniques.
%\subsection{Comments and related work}
% Notice that, using Markov's inequality, one can deduce from \eqref{eq:ineqbetti2} a simple useful estimate for the probability of having extremal configurations:
% \be \mathbb{P}\big\{b(j^{r}\psi^{-1}(W)\geq c d^{m}\big\}\leq O(d^{-m/2}).\ee
% In fact such probability is exponentially small, as discussed in \cite{DiattaLerario} and \cite{BKL}. 
\subsection{Acknowledgements} The author wish to thank the referee for her/his careful reading and her/his helpful and constructive comments.
 \section{Quantitative bounds, the h-principle and the topology semicontinuity}
 \subsection{Stratified Morse Theory}\label{sec:strat}
 
%We refer to the book \cite{GoreskyMacPherson} for the theory of a Whitney 
Let us fix a Whitney stratification $W=\sqcup_{S\in \mathscr{S}}S$ (see \cite[p. 37]{GoreskyMacPherson} for the definition) of the semialgebraic subset $W\subset J^r(S^m,\R^k)=:J$, with each stratum $S\in \mathscr{S}$ being semialgebraic and smooth (such decomposition  exists \cite[p. 43]{GoreskyMacPherson}), so that, by definition a smooth map $f\colon M\to J$, is transverse to $W$ if $f\transv S$ for all strata $S\in \mathscr{S}$. 
 When this is the case, we write $\psi\transv W$ and implicitly consider the subset $\psi^{-1}(W)\subset M$ to be equipped with the Whitney stratification given by $\psi^{-1}\mathscr{S}=\{\psi^{-1}(S)\}_{S\in\mathscr{S}}$.
 \begin{defi}\label{def:mors}
Given a Whitney stratified subset $Z=\cup_{i\in I}S_i$ of a smoooth manifold $M$ (without boundary), we say that a function $g\colon Z\to \R$ is a Morse function if $g$ is the restriction of a smooth function $\tilde{g}\colon M\to \R$ such that
\begin{enumerate}[(a)]
\item $g|_{S_i}$ is a Morse function on $S_i$. 
\item For every critical point $p\in S_i$ and every generalized tangent space $Q\subset T_pM$ (defined as in \cite[p. 44]{GoreskyMacPherson})
we have $d_p\tilde{g}(Q)\neq 0$, except for the case $Q=T_pS_i$.
\end{enumerate}
\end{defi}
Note that the condition of being a Morse function on a stratified space $Z\subset M$ depends on the given stratification of $Z$.
\begin{remark}
The definition above is slightly different than the one given in the book \cite[p. 52]{GoreskyMacPherson} by Goresky and MacPherson, where a Morse function, in addition, must be proper and have distinct critical values.
\end{remark}
%\begin{thm}\label{thm:strat}
%Let $W\subset J$ be a semialgebraic subset of a real algebraic smooth manifold $J$, with a given semialgebraic Whitney stratification $W=\sqcup_{S\in \mathscr{S}}S$. Let $M$ be a real algebraic smooth manifold (without boundary) and let $\psi\colon M\to J$, $g\colon M\to \R$ be smooth maps.
%\begin{enumerate}
%\item 
%There is a semialgebraic subset $\hat{W}\subset J^{1}(M,J\times \R)$ equipped with a semialgebraic Whitney stratification such that if $j^1(\psi,g)\transv \hat{W}$ then $\psi\transv W$ and $g|_{\psi^{-1}(W)}$ is a Morse function with respect to the stratification $\psi^{-1}\mathscr{S}$. In this case
%\be 
%\text{Crit}(g|_{\psi^{-1}(W)})=\left(j^{1}(\psi,g)\right)^{-1}(\hat{W}).
%\ee
%
%\item
%There is a constant $C>0$ depending only on $W$, such that if $\psi\transv W$, $\psi^{-1}(W)$ is compact and $g|_{\psi^{-1}(W)}$ is a Morse function, then
%\be 
%|b(\psi^{-1}(W))|\le C\sum_{i\in I}\#\text{Crit}(g|_{\psi^{-1}(S_i)}).
%\ee
%\end{enumerate}
%\end{thm}
The following theorem is the quantitative version of stratified Morse theory for semialgebraic maps we need in order to prove Theorem \ref{thm:bettibound}.
\begin{thm}\label{thm:strat}
Let $W\subset J$ be a semialgebraic subset of a real algebraic smooth manifold $J$, with a given semialgebraic Whitney stratification $W=\sqcup_{S\in \mathscr{S}}S$ and let $M$ be a real algebraic smooth manifold. 
There exists a semialgebraic subset $\hat{W}\subset J^{1}(M,J\times \R)$ having codimension larger or equal than $\dim M$, equipped with a semialgebraic Whitney stratification that satisfies the following properties with respect to any couple of smooth maps $\psi\colon M\to J$ and $g\colon M\to \R$.
\begin{enumerate}
\item 
If $\psi\transv W$and $j^1(\psi,g)\transv \hat{W}$, then $g|_{\psi^{-1}(W)}$ is a Morse function with respect to the stratification $\psi^{-1}\mathscr{S}$ and
\be \label{eq:stratcrit}
\text{Crit}(g|_{\psi^{-1}(W)})=\left(j^{1}(\psi,g)\right)^{-1}(\hat{W}).
\ee

\item
There is a constant $N_W>0$ depending only on $W$ and $\mathscr{S}$, such that if $\psi^{-1}(W)$ is compact, $\psi\transv W$ and $j^1(\psi,g)\transv \hat{W}$, then
\be 
b_i(\psi^{-1}(W))\le N_W\#\text{Crit}(g|_{\psi^{-1}(W)}),
\ee
for all $i=0,1,2\dots$
\end{enumerate}
\end{thm}

\begin{proof}
Let $S\in \mathscr{S}$ be a stratum of $W$, hence $S\subset J$ is a smooth submanifold and since $\psi\transv W$ implies that $\psi\transv S$, we also have that $\psi^{-1}(S)$ is a submanifold of $M$ of the same codimension which we denote by $k$. Define
\be 
\begin{aligned}
\hat{S}&=\{j^1_p(F,f)\in J^1(M,J\times\R)\colon F(p)\in S \text{ and }d_pf\in d_pF^*(T_{F(p)}S^\perp)\}\\
&=\{j^1_p(F,f)\in J^1(M,J\times\R)\colon F(p)\in S \text{ and }\exists \lambda\in T_{F(p)}S^\perp \text{ s.t. } d_pf=\lambda\circ d_pF\}.
\end{aligned}
\ee
Orthogonality here is meant in the sense of dual vector spaces: if $Q\subset T$ are vector spaces, then $Q^\perp=\{\xi\in T^*\colon \xi(Q)=0\}$.

It is clear, by this definition, that $\hat{S}$ is semialgebraic and its codimension is equal to the dimension of $M$. 
\begin{claim}\label{claim:stratproof1}
$j^1_{p_0}(\psi,g)\in \hat{S}$ if and only if $p_0$ is a critical point for $g|_{\psi^{-1}(S)}$.
\end{claim}
If $j^1_{p_0}(\psi,g)\in \hat{S}$, then of course $p_0\in \psi^{-1}(S)$ and there exists a (Lagrange multiplier) conormal covector $\lambda\in T_{\psi(p_0)}S^\perp$ such that $d_{p_0}g=\lambda\circ d_{p_0}\psi$. It follows that $d_{p_0}g$ vanishes on $T_{p_0}\psi^{-1}(S)=d_{p_0}\psi^{-1}(T_{\psi(p_0)}S)$. This proves the ``only if'' statement of the Claim as a consequence of the following inclusion
\be 
d_{p_0}\psi^*\left(T_{p_0}S^\perp\right)\subset\left(T_{p_0}\psi^{-1}(S)\right)^\perp.
\ee

To conclude the proof of Claim \ref{claim:stratproof1} we need to show the opposite inclusion. We do this by showing that the dimensions of the two spaces are equal. First observe that, since by hypotheses $\psi\transv S$, the image $d_{p_0}\psi$ is a complement to $T_{\psi(p_0)}S$ in $T_{\psi(p_0)}J$ and this is equivalent (it is the dual statement) to say that the restriction of $d_{p_0}\psi^*$ to $(T_{\psi(p_0)}S)^\perp$ is injective. It follows that
\be 
\begin{aligned}
\dim d_{p_0}\psi^*\left(T_{p_0}S^\perp\right)
&=
\dim \left(T_{p_0}S^\perp\right)=\\
&=
\mathrm{codim}\ S =\\
&=\mathrm{codim}\ \psi^{-1}(S)=\\
&= \dim \left(T_{p_0}\psi^{-1}(S)\right)^\perp.
\end{aligned}
\ee
This concludes the proof of Claim \ref{claim:stratproof1}.

\begin{claim}\label{claim:stratproof2}
Given a Whitney stratification of $\hat{S}$, and a critical point $p_0\in M$ of the map $g|_{\psi^{-1}(S)}$, if $j^1(\psi,g)\transv \hat{S}$ at $p_0$ then this critical point is Morse.
\end{claim}
%\begin{proof}[Proof of Lemma \ref{lem:stratproof2}]
Let us pass to a coordinate chart $\phi$ defined on a nighborhood $\mathcal{U}\subset J^1(M,J\times \R)$ of $j^1_{p_0}(\psi,g)$:
\be 
\phi=\left(
x=\begin{pmatrix}
x^1 \\ x^2
\end{pmatrix},
y=\begin{pmatrix}
y^1 \\ y^2
\end{pmatrix},
a,
Y=\begin{pmatrix}
Y^1 \\ Y^2
\end{pmatrix},
A
\right)
\colon \mathcal{U}\to \R^m\times \R^{s+k}\times \R\times \R^{(s+k)\times m}\times \R^m
\ee
\be 
j^1_p(F,f)\mapsto \left(x(p),y(F(p)),g(p),\frac{\de (y\circ F)}{\de x},\frac{\de g}{\de x}\right);
\ee
where $y^2=0$ is a local equation for $S$ and $x^2=0$ is a local equation for $\psi^{-1}(S)$. Indeed, by the implicit function theorem (applied to the map $\psi$ in virtue of the transversality assumption $\psi\transv S$) we can assume that $y^2(\psi(x^1,x^2))=x^2.$ In this coordinate chart we have that the restriction of $d_{p}F^*$ to the space ${T_{\psi(p)}S^\perp}$ is represented by the matrix $({Y}^2)^T$, thus 
\be \label{eq:shat}
\hat{S}\cap \mathcal{U}=\left\{y^2=0; A\in \mathrm{Im}\left((Y^2)^T\right)\right\}\cap \phi(\mathcal{U}).
\ee
Let us denote by $x\mapsto(x,\tilde{y}(x),\tilde{a}(x),\tilde{Y}(x),\tilde{A}(x))$ the local expression of the jet map $p\mapsto j^1_p(\psi,g)$ with respect to the above coordinates.
By construction we have that \be \left(\tilde{Y}^2(p_0)\right)^T=\begin{pmatrix} 0 \\ \mathbbm{1}_k\end{pmatrix}.\ee In particular the image of the above matrix is a complement to the subspace spanned by the first $m-k$ coordinates
and we may assume, reducing the size of the neighborhood if needed, that this property holds for every element 
$(x,y,a,Y,A)\in\phi(\mathcal{U})$, so that there exist unique vectors $\lambda\in\R^k$ and $\xi\in\R^{(m-k)}$ such that
\be \label{eq:cheduecoioni}
A=\begin{pmatrix}
A_1 \\ A_2
\end{pmatrix}=\left(Y^2\right)^T \lambda+\begin{pmatrix}
\xi\\ 0
\end{pmatrix}.
%\begin{pmatrix}(Y^2_1)^T \\(Y^2_2)^T\end{pmatrix}\begin{pmatrix}
%0\\ \lambda
%\end{pmatrix}+\begin{pmatrix}
%\xi\\ 0
%\end{pmatrix}.
\ee 
Now, this defines a smooth function $\xi\colon\mathcal{U}\to \R^k$ such that the equations $y^2=0$; $\xi=0$ are smooth regular equations for $\phi(\hat{S}\cap\mathcal{U})$. 

Notice that this ensures that $\phi(\mathcal{U})$ intersects only the smooth locus of $\hat{S}$. Now, since by hypotheses $j^1(\psi,g)$ is transverse to all the strata of $\hat{S}$ then it must be transverse to the smooth locus in the usual sense, even if the latter is a union of strata (this follows directly from the definition of transversality). Therefore, while proving Claim \ref{claim:stratproof2}, we are allowed to forget about the stratification of $\hat{S}$ and just assume that the map $j^1(\psi,g)$ is transverse to the smooth manifold $\hat{S}\cap \phi(\mathcal{U})$ in the usual sense.

In this setting we can see that if $j^1(\psi,g)\transv \hat{S}$ at $p_0$, then the following matrix has to be surjective:
\be 
\begin{pmatrix}
d y^2 \\ d\xi
\end{pmatrix}\circ d_{p_0}\left(j^1(\psi,g)\right)=
%\begin{pmatrix}
%\frac{\de\tilde{y}^2}{\de x^1} & \frac{\de\tilde{y}^2}{\de x^2} \\
%\frac{\de\tilde{\xi}}{\de x^1} &
%\frac{\de\tilde{\xi}}{\de x^2}
%\end{pmatrix}(p_0)=
\begin{pmatrix}
0 & \mathbbm{1}_k \\
\frac{\de\tilde{\xi}}{\de x^1}(p_0) &
\frac{\de\tilde{\xi}}{\de x^2}(p_0)
\end{pmatrix}\in \R^{(k+k)\times ((m-k) +k)},
\ee
where $\tilde{\xi}(x)=\xi(x,\tilde{y}(x),\tilde{a}(x),\tilde{Y}(x),\tilde{A}(x))$. 
Therefore the lower left block $\frac{\de\tilde{\xi}}{\de x^1}(p_0)$ is surjective as well and hence invertible. This concludes our proof of Claim \ref{claim:stratproof2} since such matrix is in fact the hessian of the map $g|_{\psi^{-1}(S)}$ at the critical point $p_0$:
\be 
\begin{aligned}
d_{p_0}\left(g|_{\psi^{-1}(S)}\right)&=
\frac{\de}{\de x^1}\Big|_{p_0}\left(\frac{\de g}{\de x^1}\right)
=\frac{\de \tilde{A}_1}{\de x^1}(p_0)= \frac{\de\tilde{\xi}}{\de x^1}(p_0).
%\left(\frac{\de \tilde{Y}^2_1}{\de x^1}(p_0)\right)^T\cdot \begin{pmatrix}
%0 \\ \tilde{\lambda}(p_0)
%\end{pmatrix}+(\tilde{Y}^2_1)^T(p_0)\begin{pmatrix}
%0 \\ \frac{\de\tilde{\lambda}}{\de x^1}(p_0)
%\end{pmatrix}+\frac{\de\tilde{\xi}}{\de x^1}(p_0).
\end{aligned}
\ee
The last equality is due to the equation \eqref{eq:cheduecoioni} combined with the observation that $\tilde{Y^2}$ is of the form $\begin{pmatrix}0 &*\end{pmatrix}$ for all $p$ in a neighborhood of $p_0$, since $\frac{\de \tilde{y}^2}{\de x_1}(p)=0$. 
%\end{proof}

At this point, Claim \ref{claim:stratproof1} and Claim \ref{claim:stratproof2} prove that, for whatever stratification of $\hat{S}$, if $j^1(\psi,g)\transv \hat{S}$ and $\psi\transv S$ then $g|_{\psi^{-1}(S)}$ is a Morse function and that its critical set coincide with the set $\left(j^1(\psi,g)\right)^{-1}(\hat{S})$, so that condition $(a)$ of Definition \ref{def:mors} is satisifed along the stratum $S$. In order to establish when $g|_{\psi^{-1}(W)}$ is a Morse function along the stratum $\psi^{-1}(S)$ on the stratified manifold $W$, in the sense of Definition \ref{def:mors}, we now need to prove a similar statement to ensure condition $(b)$.

Let us consider the set $D_qS$ of degenerate covectors at a point $q\in S$ that are conormal to $S$ (conormal and degenerate covectors are defined as in \cite[p.44]{GoreskyMacPherson}), in other words:
\be 
D_qS=\{\xi\in T^*_q{J}\colon \xi\in T_qS^\perp,\  \xi\in Q^\perp \text{ for some $Q$ generalized tangent space at $q$}\}. 
\ee
It is proved in \cite[p.44]{GoreskyMacPherson} that $DS=\cup_{q\in S}D_qS$ is a semialgebraic subset of codimension greater than $1$ of the conormal bundle $TS^\perp$\footnote{$TS^\perp =T_S^*J$ , in the notation of \cite{GoreskyMacPherson}.} to the stratum S.
We claim that the subset $D\hat{S}\subset \hat{S}$ containing the jets that do not satisfy condition $(b)$ of Definition \ref{def:mors} has the following description:
\be 
D\hat{S}=\{j^1_p(F,f)\in J^1(M,J\times\R)\colon F(p)\in S \text{ and }d_pf\in d_pF^*(D_{F(p)}S)\}.
\ee
In fact, since $\psi\transv W$, then all the generalized tangent spaces of the stratified subset $\psi^{-1}(W)\subset M$ at a point $p\in \psi^{-1}(S)$ are of the form $d_p\psi^{-1}(Q)$. It follows that if a conormal covector $d_pg=\lambda \circ d_p\psi$ is degenerate then $\lambda\in D_{\psi(p)}S$.

Note that $D\hat{S}$ is a subset of $\hat{S}$ of codimension $\ge 1$, thus the codimension of $D\hat{S}$ in $J^1(M,J\times \R)$ is $\ge m+1$. As a consequence we have that $j^1(\psi,g)\transv D\hat{S}$ if and only if $ j^1(\psi,g)\notin D\hat{S}$.
Therefore if $j^1(\psi,g)\transv \hat{S}$ and $j^1(\psi,g)\notin D\hat{S}$ then $\psi\transv S$ and $g|_{\psi^{-1}(W)}$ is a Morse function on $\psi^{-1}(W)$ along the stratum $\psi^{-1}(S)$.

We are now ready to define $\hat{W}=\cup_{S\in\mathscr{S}}\hat{S}$. An immediate consequence of Claim \ref{claim:stratproof1} is that $\hat{W}$ satisfies equation \eqref{eq:stratcrit}. Moreover, since $\hat{S}\supset D\hat{S}$ are semialgebraic, $\hat{W}$ is semialgebraic and admits a semialgebraic Whitney stratification $\hat{\mathscr{S}}$ (refining the one of $\hat{S}$) such that all the subsets $\hat{S}$ and $D\hat{S}$ are unions of strata. With such a stratification, if the jet map $j^1(\psi,g)$ is transverse to $\hat{W}$ then, for each stratum $S\in \mathscr{S}$, it is also transverse to $\hat{S}$ and it avoids the set $D\hat{S}$, so that $g|_{\psi^{-1}(W)}$ is a Morse function, in the sense of Definition \ref{def:mors}.
This proves that $\hat{W}$ satisfies condition $(1)$ of the Theorem.

Let us prove condition (2). Let $Z=\psi^{-1}(W)\subset M$ be compact. Without loss of generality we can assume that each of the critical values $c_1,\dots, c_n$ of $g|_{Z}$ corresponds to only one critical point (this can be obtained by makingcontaining the jets that do not satisfy condition $(b)$ of Definition \ref{def:mors}: a $\mathcal{C}^1$ small perturbation of $g$, which won't affect the number of its critical points). Consider a sequence of real numbers $a_1,\dots a_{n+1}$ such that
\be 
a_1<c_1 <a_2<c_2<\dots <a_{n}<c_n<a_{n+1}.
\ee
By the main Theorem of stratified Morse theory \cite[p. 8, 65]{GoreskyMacPherson}, there is an homeomorphism
\be 
Z\cap \{g\le a_{l+1}\}\cong(Z\cap \{g\le a_l\})\sqcup_B A,
\ee
with
\be 
(A,B)= TMD_p(g)\times NMD_p(g),
\ee
where $TMD_p(g)$ is the tangential Morse data and $NMD_p(g)$ is the normal Morse data. A fundamental result of classical Morse theory is that the tangential Morse data is homeomorphic to a pair 
\be 
TMD_p(g)\cong(\mathbb{D}^\lambda\times \mathbb{D}^{m-\lambda},(\de\mathbb{D}^\lambda)\times \mathbb{D}^{m-\lambda}),
\ee 
while the normal Morse data is defined as the local Morse data of $g|_{N_p}$ for a normal slice (see \cite[p. 65]{GoreskyMacPherson}) at $p$. A consequence of the transversality hypothesis $\psi\transv W$ is that there is a small enough normal slice $N_p$ such that $\psi|_ {N_p}\colon N_p\to J$ is the embedding of a normal slice at $\psi(p)$ for $W$. Therefore the normal data $NMD_p(g)$ belongs to the set $\nu(W)$ of all possible normal Morse data that can be realized (up to homeomeorphisms) by a critical point of a Morse function on $W$. By Corollary $7.5.3$ of \cite[p. 95]{GoreskyMacPherson} it follows that the cardinality of the set $\nu(W)$ is smaller or equal than the number of connected components of the semialgebraic set $\cup_{S\in \mathscr{S}}(TS^\perp\-DS)$, hence finite\footnote{In the book this is proved only for any fixed point $p$, as a corollary of Theorem $7.5.1$ \cite[p.93]{GoreskyMacPherson}. However the same argument generalizes easily to the whole bundle.}.
Let 
\be 
N_W:=\max_{Y \in \nu(W),\ \lambda \in \{0,\dots, m\}} b_i\left(\left(\mathbb{D}^\lambda\times \mathbb{D}^{m-\lambda},(\de\mathbb{D}^\lambda)\times \mathbb{D}^{m-\lambda}\right)\times Y\right)\in \N.
\ee
%Qui credo che ci si pu√≤ restringere ai lambda che dividono i.
From the long exact sequence of the pair $\left(Z\cap \{g\ge a_{l+1}\},(Z\cap \{g\ge a_{l}\}\right)$ we deduce that
\be\label{eq:lesbetti}\begin{aligned}
b_i(Z\cap \{g\le a_{l+1}\})-b_i(Z\cap \{g\le a_{l}\})&\le b_i\left(Z\cap \{g\le a_{l+1}\},Z\cap \{g\le a_{l}\}\right)\\
&= b_i\left(A,B\right) \\
&= b_i\left(TMD_p(g)\times NMD_p(g)\right) \\
&\le N_W.
\end{aligned}
\ee
Since $Z$ is compact, the set $Z\cap \{g\le a_1\}$ is empty, hence by repeating the inequality \eqref{eq:lesbetti} for each critical value, we finally get
\be 
b_i(Z)=b_i(Z\cap g\le a_{n+1})\le N_Wn=N_W\#\text{Crit}\left(g|_{\psi^{-1}(W)}\right).
\ee
This concludes the proof of Theorem \ref{thm:strat}.
\end{proof}
Below we will restrict to those semialgebraic sets $W\subset J^{r}(S^m, \R^k)$ that have a differential geometric meaning, as specified in the next definition.
\begin{defi}\label{def:intrinsic}
A submanifold $W\subset J^r(M,\R^k)$ is said to be \emph{intrinsic} if there is a submanifold $W_0\subset J^r(\mathbb{D}^m,\R^k)$, called the \emph{model}, such that for any embedding $\f\colon \mathbb{D}^m\hookrightarrow M$, one has that $j^r\f^*(W)=W_0$, where 
\be 
j^r\f^*\colon J^r\left(\f(\mathbb{D}^m),\R^k\right)\xrightarrow{\cong}J^r\left(\mathbb{D}^m,\R^k\right), \qquad j^r_{\f(p)}f\mapsto j^r_p(f\circ\f ).
\ee
\end{defi}
Intrinsic submanifolds are, in other words, those that have the same shape in every coordinate charts, as in the following examples.
\begin{enumerate}
    \item $W=\{j^r_pf\colon f(p)=0\}$;
    \item $W=\{j^r_pf\colon j^sf(p)=0\}$ for some $s\le r$;
    \item $W=\{j^r_pf\colon \text{rank}(df(p))=s\}$ for some $s\in\N$.
\end{enumerate}
\begin{remark}\label{rem:strat}
In the case when $J=J^{r}(M,\R^k)$ we can consider $\hat{W}$ to be a subset of $ J^{r+1}(M,\R^{k+1})$ taking the preimage via the natural submersion
\be 
J^{r+1}(M,\R^{k+1})\to J^1\left(M,J^{r}(M,\R^k)\times\R \right), \qquad j^{r+1}(f,g)\mapsto j^1(j^rf,g).
\ee
In this setting Theorem \ref{thm:strat} can be translated to a more natural statement by considering $\psi$ of the form $\psi=j^rf$.
Moreover, in this case, observe that if $W$ is intrinsic (in the sense of Definition \ref{def:intrinsic} below), then $\hat{W}$ is intrinsic as well.
\end{remark}

\subsection{Quantitative bounds}\label{sec:quantitative}
In this section we prove Theorem \ref{thm:bettibound}, which actually immediately follows by combining Theorem \ref{thm:bound} and Theorem \ref{thm:bound2}.

%Let $W\subset J^r(S^m,\R^k)$ be a semialgebraic subset: $W$ can be written as:
%\be W=\bigcup_{j=1}^\ell\left\{f_{j, 1}=0, \ldots, f_{j, \alpha_j}=0, g_{j, 1}>0,\ldots, g_{j, \beta_j}>0\right\},\ee
%where the $f_{j,i}$s and the $g_{j, i}$s are polynomial functions on the space $J^r(\R^{m+1},\R^k)$, where $J^r(S^m,\R^k)$ is naturally embedded as a smooth algebraic submanifold. 

Next theorem gives a deterministic bound for on the complexity of $Z=j^r\psi^{-1}(W)$ when the codimension of $W$ is $m$.
\begin{thm}\label{thm:bound}
Let $P\in \R[x_0, \ldots, x_m]_{(d)}^k$ be a polynomial map and consider its restriction $\psi=P|_{S^m}$ to the unit sphere:
\be \psi:S^m\to \R^k.\ee
Let also $j^r\psi:S^m\to J^r(S^m, \R^k)$ be the associated jet map and $W\subset J^{r}(S^m, R^k) $ be a semialgebraic set of codimension $m$. There exists a constant $c>0$ (which only depends on $W$, $m$ and $k$) such that, if $j^r\psi\transv W$,  then:
\be\label{eq:detbound} \#j^r\psi^{-1}(W)\leq c \cdot d^m.\ee
\end{thm}
\begin{proof}
Let us make the identification $J^{r}(\R^{m+1}, \R^k)\simeq \R^{m+1}\times \R^N$, so that the restricted jet bundle $J^{r}(\R^{m+1}, \R^k)|_{S^m}$ corresponds to the semialgebraic subset $S^m\times \R^N$. Observe that the inclusion $S^m\hookrightarrow \R^{m+1}$ induces a semialgebraic map:
\be J^{r}(\R^{m+1}, \R^k)|_{S^m}\stackrel{i^*}{\longrightarrow}  J^r(S^m, \R^k),\ee
that, roughly speaking, forgets the normal derivatives.
Notice that while the map $j^{r}\psi=j^r(P|_{S^m})$ is a section of $J^r(S^m, \R^k)$,  $(j^rP)|_{S^m}$ is a section of $J^{r}(\R^{m+1}, \R^k)|_{S^m}$. These sections are related by the identity \be 
 i^*\circ (j^rP)|_{S^m}=j^{r}\psi.
 \ee
%These sections fit into a commutative diagram: 
%\be\begin{tikzcd}
%{J^r(\mathbb{R}^{m+1}\backslash\{0\}, \mathbb{R}^k)|_{S^m}} \arrow[rr, "j^*"] \arrow[rddd, "p_1"] &  & {J^r(S^m, \mathbb{R}^k)} \arrow[lddd, "p_2"'] \\
% &  &  \\
% &  &  \\
% & S^m \arrow[ruuu, "j^r\psi"', bend right] \arrow[luuu, "j^rP|_{S^m}", bend left] & 
%\end{tikzcd}
%\ee
%where all spaces and maps are semialgebraic. 
Thus, defining $ \overline{W}=(i^*)^{-1}(W)$, we have
%\be \overline{W}=(j^*)^{-1}(W)\subset J^{r}(\R^{m+1}\backslash\{0\}, \R^k)|_{S^m}\subset S^{m}\times \R^N\subset \R^{m+1}\times \R^N.\ee 
%With this choice we have:
\be j^r\psi^{-1}(W)=\left((j_rP)|_{S^m}\right)^{-1}(\overline W).\ee
Since $\overline{W}$ is a semialgebraic subset of $\R^{m+1}\times \R^N$ , it can be written as:
\be \overline{W}=\bigcup_{j=1}^\ell\left\{f_{j, 1}=0, \ldots, f_{j, \alpha_j}=0, g_{j, 1}>0,\ldots, g_{j, \beta_j}>0\right\},\ee
where the $f_{j,i}$s and the $g_{j, i}$s are polynomials of degree bounded by a constant $b>0.$ For every $j=1, \ldots, \ell$ we can write:
\be \left\{f_{j, 1}=0, \ldots, f_{j, \alpha_j}=0, g_{j, 1}>0,\ldots, g_{j, \beta_j}>0\right\}=Z_j\cap A_j,\ee
where $Z_j$ is algebraic (given by the equations) and $A_j$ is open (given by the inequalities).

Observe also that the map $(j^rP)|_{S^m}$ is the restriction to the sphere $S^m$ of a polynomial map
\be Q:\R^{m+1}\to \R^{m+1}\times \R^N\ee
whose components have degree smaller than $d$.
Therefore for every $j=1\ldots, \ell$ the set $((j^rP)|_{S^m})^{-1}(Z_j)=(Q|_{S^m})^{-1}(Z_j)$ is an algebraic set on the sphere defined by equations of degree less than $b \cdot d$ and, by \cite[Proposition 14]{LerarioJEMS} we have that:
\be\label{eq:ineqb} b_0(Q|_{S^m})^{-1}(Z_j))\leq B d^m\ee
for some constant $B>0$ depending on $b$ and $m$.
The set $(Q|_{S^m})^{-1}(Z_j)$ consists of several components, some of which are zero dimensional (points):
\be (Q|_{S^m})^{-1}(Z_j)=\underbrace{\{p_{j, 1}, \ldots, p_{j, \nu_j}\}}_{P_j}\cup \underbrace{X_{j,1}\cup\cdots \cup X_{j, \mu_j}}_{Y_j}.\ee
The inequality \eqref{eq:ineqb} says in particular that: \be\label{ineqbb}\#P_j\leq Bd^n.\ee
Observe now that if $j^r\psi\transv W$ then, because the codimension of $W$ is $m$, the set $j^{r}\psi^{-1}(W)=(Q|_{S^m})^{-1}(\overline{W})$ consists of finitely many points and therefore, since $(Q|_{S^m})^{-1}(A_j)$ is open, we must have:
\be j^{r}\psi^{-1}(W)\subset \bigcup_{j=1}^\ell P_j.\ee
(Otherwise $ j^{r}\psi^{-1}(W)$ would contain an open, nonempty set of a component of codimension smaller than $m$.)
Inequality \eqref{ineqbb} implies now that:
\be \#j^{r}\psi^{-1}(W)\leq \sum_{j=1}^\ell \#P_j\leq \ell b d^m\leq cd^m.\ee
\end{proof}

Using Theorem \ref{thm:strat} it is now possible to improve Theorem \ref{thm:bound} to the case of any codimension, replacing the cardinality with any Betti number. \begin{thm}\label{thm:bound2}
Let $P\in \R[x_0, \ldots, x_m]_{(d)}^k$ be a polynomial map and consider its restriction $\psi=P|_{S^m}$ to the unit sphere:
\be \psi:S^m\to \R^k.\ee
Let also $j^r\psi:S^m\to J^r(S^m, \R^k)$ be the associated jet map and $W\subset J^{r}(S^m, R^k) $ be a closed semialgebraic set (of arbitrary codimension). There exists a constant $c>0$ (which only depends on $W$, $m$ and $k$) such that, if $j^r\psi\transv W$,  then:
\be b_i\left(j^r\psi^{-1}(W)\right)\leq c \cdot d^m.\ee
\end{thm}
\begin{proof}
Let $J=J^r(S^m,\R^k)$ and let $\hat{W}$ be the (stratified according to a chosen stratification of $W$) subset of $J^{r+1}(S^m,\R^{k+1})$ coming from Theorem \ref{thm:strat} and Remark \ref{rem:strat}.
Let $g$ be a homogeneous polynomial of degree $d$ such that 
\be 
\Psi=(\psi,g)\in \R[x_0, \ldots, x_m]_{(d)}^{k+1}
\ee
satisfies the condition $j^{r+1}\Psi \transv \hat{W}$ (almost every polynomial $g$ has this property by standard arguments) and $(j^{r}\psi)^{-1}(W)$ is closed in $S^m$, hence compact. Then by Theorem \ref{thm:strat}, there is a constant $N_W$, such that
\be 
b_i\left(j^r\psi^{-1}(W)\right)\le N_W \#\{(j^{r+1}\Psi)^{-1}(\hat{W})\}
\ee
and by Theorem \ref{thm:bound}, the right hand side is bounded by $cd^m$.
\end{proof}
Given $P=(P_1, \ldots, P_k)$ with each $P_i$ a homogeneous polynomial of degree $d$ in $m+1$ variables, we denote by
\be \psi_d:S^m\to \R^k
\ee
its restriction to the unit sphere (the subscript keeps track of the dependence on $d$).
\begin{example}[Real algebraic sets] Let us take $W=S^m\times \{0\}\subset J^{0}(S^m, \R^k),$ then $j^0\psi^{-1}(W)$ is the zero set of $\psi_d:S^m\to \R^k$, i.e. the set of solutions of a system of polynomial equations of degree $d$. In this case the inequality \eqref{eq:detbound} follows from \cite{LerarioJEMS}.
\end{example}
\begin{example}[Critical points] If we pick $W=\{j^1f=0\}\subset J^1(S^m, \R),$ then $Z_d=j^1\psi_d^{-1}(W)$ is the set of critical points of $\psi_d:S^m\to \R$. In $2013$ Cartwright and Sturmfels \cite{CS} proved that 
\be 
\#Z_d\le 2(d-1)^m+\dots+(d-1)+1
\ee
(this bounds follows from complex algebraic geometry), and this estimate was recently proved to be sharp by Kozhasov \cite{khazOFRET}.
Of course one can also fix the index of a nondegenerate critical point (in the sense of Morse Theory); for example we can take $W=\{df=0, d^2f>0\}\subset J^2(S^m, \R),$ and $j^2\psi_d^{-1}(W)$ is the set of nondegenerate \emph{minima} of $\psi_d:S^m\to \R$ (similar estimates of the order $d^{m}$ holds for the fixed Morse index, but the problem of finding a sharp bound is very much open).
\end{example}

\begin{example}[Whitney cusps]When $W=\{\textrm{Whitney cusps}\}\subset J^3(S^2, \R^2),$ then $\psi_d^{3}f^{-1}(W)$ consists of the set of points where the polynomial map $\psi_d:S^2\to \R^2$ has a critical point which is a Whitney cusp. In this case \eqref{eq:detbound} controls the number of possible Whitney cusps (the bound is of the order $O(d^2)$).
\end{example}

\subsection{Semicontinuity of topology under holonomic approximation}\label{sec:holo}
Consider the following setting: $M$ and $J$ are smooth manifolds, $M$ is compact, and $W\subset J$ is a smooth cooriented submanifold. Given a smooth map $F\colon M\to J$ which is transversal to $W$, it follows from standard transversality arguments that there exists a small $\mathcal{C}^1$ neighborhood $U_1$ of $F$ such that for every map $\tilde {F}\in U_1$ the pairs $(M, F^{-1}(W))$ and $(M, \tilde{F}^{-1}(W))$ are isotopic (in particular $F^{-1}(W)$ and $\tilde{F}^{-1}(W)$ have the same Betti numbers, this is the so-called ``Thom's isotopy Lemma''). The question that we address is the behavior of the Betti numbers of $\tilde{F}^{-1}(W)$ under small $\mathcal{C}^0$ perturbations, i.e. how the Betti number can change under modifications of the map $F$ \emph{without} controlling its derivative.  
\begin{figure}\begin{center}
\includegraphics[scale=0.11]{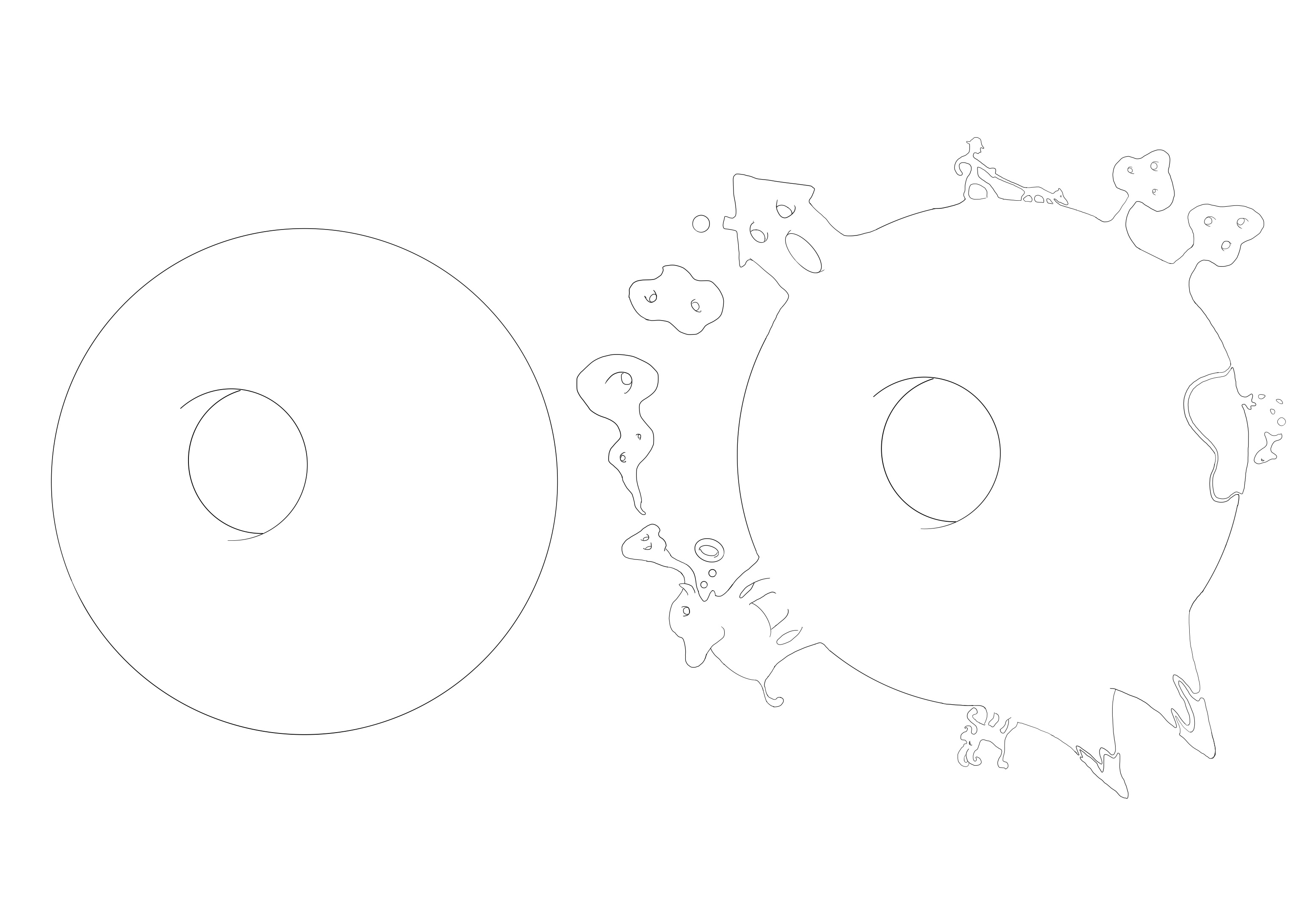}\caption{A small $\mathcal{C}^0$ perturbation of a regular equation can only increase the topology of its zero set.}\label{fig:perturb}
\end{center}\end{figure}
In this direction we prove the following result.
\begin{thm}\label{thm:semiconttop}
Let $M,J$ be smooth manifolds and let $W\subset J$ be a smooth cooriented closed submanifold.
Let $F\colon M\to J$ be a smooth map such that $F\transv W$. If a smooth map $\tilde{F}$ is strongly\footnote{Meaning: in Whitney strong topology. In particular if $C\subset M$ is closed and $U\subset J$ is open, then the set $\{f\in \mathcal{C}^0(M,J)\colon f(C)\subset U\}$ is open, see \cite{Hirsch}.} $\mathcal{C}^0-$close to $F$ such that $\tilde{F}\transv W$, then for all $i\in\N$ there is a group $K^i$ such that
\be \label{eq:semiconttop}
H^i\left(\tilde{F}^{-1}(W)\right)\cong H^i\left(F^{-1}(W)\right)\oplus K^i.
\ee
\end{thm}

%\begin{thm}
%Let $M,J$ be smooth manifolds and let $W\subset J$ be a smooth cooriented submanifold.
%Let $F\colon M\to N$ be a smooth map such that $F\transv W$. If a smooth map $\tilde{F}$ is strongly \footnote{Meaning: in Whitney strong topology. In particular if $C\subset M$ is closed and $U\subset J$ is open, then the set $\{f\in \mathcal{C}^0(M,J)\colon f(C)\subset U\}$ is open.} $\mathcal{C}^0-$close to $F$ such that $\tilde{F}\transv W$, then there is an algebra isomorphism
%\be 
%H^*\left(\tilde{F}^{-1}(W)\right)\cong H^*\left(F^{-1}(W)\right)\oplus K
%\ee
%for some algebra $K$.
%\end{thm}
\begin{proof}Call $A=F^{-1}(W)$ and $\tilde{A}=\tilde{F}^{-1}(W)$. Let
$E\subset M$ be a closed  tubular neighborhood (it exists because $A$ is closed), meaning that $E=\text{int}(E)\cup\de E$ is diffeomorphic to the unit ball of a metric vector bundle over $A$ (via a diffeomorphism that preserves $A$). Denote by $\pi\colon E\to A$ the retraction map. Since $\tilde{F}$ is $\mathcal{C}^0-$close to $F$ we can assume that there is a homotopy $F_t$ connecting $F=F_0$ and $\tilde{F}=F_1$ such that $F_t(\de E)\subset J\-
W$.
 Define analogously $\tilde{\pi}:\tilde{E}\to \tilde{A}$ in such a way that $\tilde{E}\subset \text{int}(E)$. It follows that there is an inclusion of pairs $u:(E,\de E)\to (E,E\-\tilde{E})$.
By construction, the function $F_t$ induces a well defined mapping of pairs $F_t\colon (E,\de E)\to (J,J\-W)$ for every $t\in[0,1]$, in particular there is a homotopy between $F_0$ and $F_1$ (meant as maps of pairs). Moreover with $t=1$, this map is the composition of $u$ and the map $F_1\colon (E,E\-\tilde{E})\to (J,J\-W)$.

The fact that $W$ is closed and cooriented guarantees the existence of a Thom class $\phi\in H^r(J,J\- W)$, where $r$ is the codimension of $W$. By transversality we have that also $A$ and $\tilde{A}$ are cooriented with Thom classes $F_0^*\phi=\phi_E\in H^r(E,\de E)\cong H^r(E,E\- A)$ and $F_1^*\phi=\phi_{\tilde{E}}\in H^r(\tilde{E},\de\tilde{E})\cong H^r(\tilde{E},\tilde{E}\- \tilde{A})$.
We now claim the commutativity of the diagram below. 
%\be 
%\begin{tikzcd}
%                                                                                               & {H^{*+r}(N, N\backslash W)}                          &                                                                  \\
%                                                                                               &                                                      &                                                                  \\
%{H^{*+r}(\tilde{E}, \tilde{E}\backslash\tilde{B})} \arrow[r, "\eta^{-1}"] \arrow[ruu, "f_1^*"] & {H^{*+r}(E, E\backslash \tilde{B})} \arrow[r, "u^*"] & {H^{*+r}(E, E\backslash B)} \arrow[luu, "f_0^*=f_1^*"']          \\
%H^{*}(\tilde{A}) \arrow[u, "\tilde{\pi}^*(\cdot)\cup \phi_{\tilde{B}}"]                        &                                                      & H^*(A) \arrow[ll, "\pi^*"] \arrow[u, "\pi^*(\cdot)\cup \phi_B"']
%\end{tikzcd}\ee
\be 
\begin{tikzcd}
                                                                                               & {H^{*+r}(J, J\backslash W)}\arrow[ldd, "F_1^*"'] \arrow[rdd, "F_1^*=F_0^*"]                         &                                                                  \\
                                                                                               &                                                      &                                                                  \\
{H^{*+r}(\tilde{E}, \de\tilde{E})} \arrow[r, "\eta^{-1}"]  & {H^{*+r}(E, E\backslash \tilde{E})} \arrow[r, "u^*"] & {H^{*+r}(E, \de E)}           \\
H^{*}(\tilde{A}) \arrow[u, "\tilde{\pi}^*(\cdot)\cup \phi_{\tilde{E}}"]                        &                                                      & H^*(A) \arrow[ll, "\pi^*"] \arrow[u, "\pi^*(\cdot)\cup \phi_E"']
\end{tikzcd}\ee
(where $\eta$ is the excision isomorphism).
For what regards the upper triangular diagram, the commutativity simply follows from the fact that all the maps $F_t$ are homotopic and that the excision homomorphism is the inverse of that induced by the inclusion $(E,E\-\tilde{E})\subset (\tilde{E},\de \tilde{E})$.
 To show that the lower rectangle commutes, observe that since $\tilde{\pi}$ is homotopic to the identity of $\tilde{E}$ we have that $\pi\circ\tilde{\pi}$ is homotopic to $\pi|_{\tilde{E}}$. Thus
the commutativity follows from the property of the cup product, saying that for all $\f\in H^*(A)$ we have
\be\begin{aligned}
u^*\circ \eta^{-1}\circ\left(\tilde{\pi}^*\left(\pi|_{\tilde{A}}\right)^*\f\right)\cup \phi_{\tilde{E}}
&=
\left(u^*\circ\eta^{-1}\circ\left(\pi|_{\tilde{E}}\right)^*\f\right)\cup\left( u^*\circ \eta^{-1} \circ F_1^*\phi\right)\\
&=
\pi^*\f\cup \phi_E,
\end{aligned}\ee
where in the last equality we used the identity $u^*\circ \eta^{-1}\circ F_1^*=F_0^*$ implied by the commutativity of the upper triangle.
Since the vertical maps are (Thom) isomorphisms, there exists a homomorphism $U\colon H^*(\tilde{A})\to H^*(A)$ such that $U\circ \pi^*=$id.
\end{proof}
\begin{remark}\label{rem:semiconttop}
The above proof also provides a way to determine how small should the perturbation be. In fact we showed that if $F_t\colon M\to J$ is a homotopy such that $F_1\transv W$ and $F_t(\de E)\subset J\- W$ for all $t\in [0,1]$, where $E$ is a closed tubular neighborhood of $F^{-1}(W)$, then the map $\tilde{F}=F_1$ satisfies \eqref{eq:semiconttop}. Notice that to have such property it is enough that $\tilde{F}\transv W$ and $\tilde{F}|_{\de E}$ is $\mathcal{C}^0-$close to $F|_{\de E}$.  
 This implies that the size of the $\mathcal{C}^0$ neighborhood of $F$ in which the identity \eqref{eq:semiconttop} holds depends only on the restriction of $F$ to a codimension $1$ submanifold.
\end{remark}
\begin{cor}\label{cor:topositive}Let $M$ be a compact manifold of dimension $m$.
Let $W\subset J^r(M,\R^k)$ be a Whitney stratified submanifold of codimension $1\le l\le m$ being transverse to the fibers of the canonical projection $\pi\colon J^r(M,\R^k)\to M$. Then for any number $n \in \N$ there exists a smooth function $\psi\in \mathcal{C}^{\infty}(M,\R^k)$ such that $j^r\psi\transv W$ and 
\be 
b_i\left((j^r\psi)^{-1}(W)\right)\ge n, \quad \forall i=0,\dots, m-l.
\ee 
\end{cor}
\begin{proof}
Let $B\subset J^r(M,\R^k)$ be a small neighbourhood of a regular point $j^r_pf$ of $W$ so that $(B,B\cap W)\cong (\R^{N+l},\R^N\times \{0\})$. Moreover we can assume that there is a neighbourhood $U\cong \R^m$ of $p\in M$ and a commutative diagram of smooth maps
\be\label{eq:diagramabove}
\begin{tikzcd}
                                                         & \R^m\times \R^k\times \{0\} \arrow[rr, hook] &                                                        & \R^m\times\R^k\times\R^l \arrow[d] \\
B\cap W \arrow[rr, hook] \arrow[ru, "\cong" description] &                                              & B \arrow[d, "\pi"'] \arrow[ru, "\cong" description] & \R^m                               \\
                                                         &                                              & U \arrow[ru, "\cong" description]                      &                                   
\end{tikzcd}
\ee
This follows from the fact that $\pi|_{W}$ is a submersion, because of the transversality assumption. 
%Let 
%\be 
%r=1+2+4+\dots+2^\ell+a
%\ee 
%with $a\le 2^{\ell+1}$ and notice that for all $i\in \{0,\dots, r\}$ we have
%\be 
%b_i\left(S^1\times S^2\times S^4\times \dots\times S^{2^\ell}\times S^a\right)\ge 1
%\ee.
%\be 
%m=r+2+3+5+\dots+(1+2^\ell)+a
%\ee
%such that $a\le 1+2^\ell$.
For any $0\le i\le m-l$ consider the smooth map
\be 
\f_i\colon \R^m\to \R^l, \quad 
u\mapsto \left(\sum_{\ell=1}^{i+1} (u_\ell)^2-1,\sum_{\ell=i+2}^{m} (u_\ell)^2-1 ,u_{m-l+3},\dots,u_m\right)
\ee 
Clearly $0$ is a regular value for $\f_i$, with preimage\footnote{Except for the case  $l=1$. Here one should adjust the definition of $\f_i$ in order to have $b_i(\f_i^{-1}(0))>0$.} $\f_i^{-1}(0)\cong S^i\times S^{m-l-i}$ and it is contained in the unit ball of radius $2$. Let $C\subset\R^m$ be a set of $n(m-l+1)$ points such that $|c-c'|\ge 5$ for all pair of distinct elements $c,c'\in C$. Now choose a partition $C=C_0\sqcup C_1\sqcup\dots C_{m-l}$ in sets of cardinality $n$ and define a smooth map $\f\colon \R^m\to \R^l$ such that $\f(x)=\f_i(x-c)$ whenever $\text{dist}(x,C_i)\le 2$. We may also assume that $0$ is a regular value for $\f$. Notice that $\f^{-1}(0)$ has a connected component 
\be 
S\cong \{1,\dots, n\}\times\left( S^0\times S^{m-l}\sqcup S^1\times S^{m-l-1}\sqcup \dots S^{m-l}\times S^{0}\right).
\ee
Construct a smooth (non necessarily holonomic) section $F\colon U\to J^r(U,\R^k)$ such that $F\transv W$ and such that $F=(u,0,\f)$ on a neighbourhood of $S$, so that $F^{-1}(W)$ still contains $S$ as a connected component, hence $b_i(F^{-1}(W))\ge n$ for all $i=0,\dots, m-l$.

Let $E\subset U$ be a closed tubular neighborhood of $F^{-1}(W)$. 
To conclude we use the holonomic approximation theorem \cite[p. 22]{eliash}, applied to $F\colon U\to J^r(U,\R^k)\cong U\times \R^{k+l}$ near the codimension $1$ submanifold $\de E\subset U$. Such theorem ensures that for any $\e>0$ there exists a diffeomorphism $h\colon U\to U$, an open neighborhood $O_{\de E}\subset U$ of $\de E$ and a smooth function $\psi \colon U\to \R^k$ such that
\be 
\text{dist}_{\mathcal{C}^0}\left((j^rf)|_{h(O_{\de E})},F|_{h(O_{\de E})}\right)<\e,\quad  \text{and }\quad  \text{dist}_{\mathcal{C}^0}\left(h,\mathrm{id}\right)<\e.
\ee
Moreover, we can assume that $j^r\psi\transv W$, by Thom transversality Theorem (see \cite{Hirsch} or \cite{eliash}). 
In particular, it follows that
\be 
\text{dist}_{\mathcal{C}^0}\left((j^rf)\circ h|_{\de E},F|_{\de E}\right)<\left(1+C(F)\right)\cdot \e,
\ee
where $C(F)$ is the lipshitz constant of $F|_{U}$, which can be assumed to be finite (if not, replace $U\cong \R^m$ with an open ball that still contains $F^{-1}(W)$).
Consider the smooth manifold $J=J^r(U,\R^k)$. By the diagram \eqref{eq:diagramabove} it follows that $W\subset J$ is a closed and cooriented smooth submanifold, so that by Theorem \ref{thm:semiconttop} and Remark \ref{rem:semiconttop} we know that if $\e>0$ is small enough, then the map $\tilde{F}=(j^rf)\circ h$ satisfies the identity \eqref{eq:semiconttop}. Therefore for each $i=0,\dots,m-l$, we have
\be 
\begin{aligned}
b_i\left(\left(j^rf\right)^{-1}(W)\right)
&=
b_i\left(\left((j^rf)\circ h\right)^{-1}(W)\right)\\
 &\ge 
b_i\left(\left(F\circ h\right)^{-1}(W)\right) \\
&=
b_i\left(F^{-1}(W)\right)\\
&\ge 
n.
\end{aligned}
\ee
\end{proof}
\section{Random Algebraic Geometry}\label{sec:RAG}
\subsection{Kostlan maps}In this section we give the definition of a random Kostlan polynomial map $P:\R^{m+1}\to \R^{k}$, which is a Gaussian Random Field (GRF) that generalizes the notion of Kostlan polynomial.
\begin{defi}[Kostlan polynomial maps]\label{def:Kostlan}
Let $d,m,k\in \N$. We define the degree $d$ homogeneous Kostlan random map as the measure on $\R[x]_{(d)}^k=\R[x_0,\dots,x_{m}]_{(d)}^k$ induced by the gaussian random polynomial:
\be 
P_d^{m,k}(x)=\sum_{\a\in \N^{m+1},\ |\a|=d} \xi_\a x^\a,
\ee
where $x^\a=x_0^{\a_0}\dots x_m^{\a_m}$ and $\{\xi_\a\}$ is a family of independent gaussian random vectors in $\R^k$ with covariance matrix
\be 
K_{\xi_\a}={d\choose\a}\mathbbm{1}_k=\left(\frac{d!}{\a_0!\dots \a_m!}\right)\mathbbm{1}_k.
\ee
We will call $P_d^{m,k}$ the \emph{Kostlan polynomial} of type $(d,m,k)$ (we will simply write $P_d=P_d^{m,k}$ when the dimensions are understood).
\end{defi}
(In other words, a Kostlan polynomial map $P_d^{m, k}$ is given by a list of $k$ independent Kostlan polynomials of degree $d$ in $m+1$ homogeneous variables.)

There is a non-homogeneous version of the Kostlan polynomial, which we denote as
\be \label{eq:Kd}
p_d(u)=P_d(1,u)=\sum_{\beta\in \N^{m},\ |\beta|\le d} \xi_\beta u^\beta  \in \g \infty{\R^m}k,
\ee
where $u=(u_1,\dots, u_m)\in \R^m$ and $\xi_\beta\sim N\left(0,{d\choose \beta}\mathbbm{1}_k\right)$ are independent. Here we use the notation of \cite{DTGRF1}, where $\g \infty{\R^m}k$ denotes the space of gaussian random field on $\R^m$ with values in $\R^k$ which are $\mathcal{C}^{\infty}$.
Next Proposition collects some well known facts on the Kostlan measure.
\begin{prop}\label{covkostlanprop}Let $P_d$ be the Kostlan polynomial of type $(d,m,k)$ and $p_d$ be its dehomogenized version, as defined in \eqref{eq:Kd}.
\begin{enumerate}
\item
 For every $x,y\in \R^{m+1}$:
\be 
K_{P_d}(x,y)=\left(x^Ty\right)^d\mathbbm{1}_k.
\ee
Moreover, given $R\in O(m+1)$ and $S\in O(k)$ and defined the polynomial $\tilde{P}_d(x)=SP_d(Rx)$, then $P_d$ and $\tilde{P}_d$ are equivalent\footnote{Two random fields are said to be \emph{equivalent} if they induce the same probability measure on $\Cr \infty{\R^m}k$.}.

\item  For every $u, v\in \R^n$
\be K_{p_d}(u,v)=(1+u^Tv)^d\mathbbm{1}_k.\ee
Moreover, if $R\in O(m)$ and $S\in O(k)$ and defined the polynomial $\tilde{p}_d(x)=Sp_d(Rx)$, then $p_d$ and $\tilde{p}_d$ are equivalent.
\end{enumerate}
\end{prop}
\begin{proof}The proof of this proposition simply follows by computing explicitly the covariance functions and observing that they are invariant under orthogonal change of coordinates in the target and the source. For example, in the case of $P_{d}$ we have:
\be
\begin{aligned}
K_{P_d}(x,y)&= \E\{P_d(x)P_d(y)^T\}= \\
&= \sum_{|\a|,|\a '|=d}\E\left\{\xi_\a\xi_{\a '}^T\right\} x^\a y^{\a'}=\\
&=\sum_{|\a|=d}{d\choose\a} (x_0 y_0)^{\a_0}\dots(x_m y_m)^{\a_m}\mathbbm{1}_k=\\
&= (x_0y_0+\dots +x_my_m)^d\mathbbm{1}_k,
\end{aligned}
\ee
from which the orthogonal invariance is clear. The case of $p_d$ follows from the identity:
\be 
K_{p_d}(u,v)=K_{P_d}\left((1,u),(1,v)\right).
\ee .
\end{proof}
%As a corollary we can derive the following useful Theorem.
%\begin{thm}
%Let $p_d\in \g \infty{\R^m}k$ be the dehomogeneized Kostlan polynomial. Then the support $\spt(p_d)$ is the space of polynomials with degree smaller than or equal to $d$, which we denote by $\R[u_1,\dots,u_m]^k_{\le d}$. In particular, given $W\subset J^r(\R^m,\R^k)$ a smooth submanifold, with $r\le d$. Then 
%\be 
%\P\{j^r p_d\pitchfork W\}=1.
%\ee
%\end{thm}
%\begin{proof} The statement about the support is clear from the definition. For the second part of the statement, we can apply Theorem \ref{transthm2} to $p_d$, since 
%\be 
%\text{span}\{j^r_pf\colon |f\in \R[u_1,\dots,u_m]^k_{\le d}\}=J^r_p(M,\R^k),
%\ee
%whenever $r\le d$.
%\end{proof}
\subsection{Properties of the rescaled Kostlan}

The main feature here is the fact that the local model of a Kostlan polynomial has a rescaling limit. The orthogonal invariance is used to prove that the limit does not depend on the point where we center the local model, hence it is enough to work around the point $(1, 0, \ldots, 0)\in S^m$. These considerations lead to introduce the Gaussian Random Field $X_d:\R^m\to\R^{k}$ (we call it the \emph{rescaled Kostlan}) defined by:
\be \label{eq:rescaledKostlan}X_d(u)=P_d^{m,k}\left(1, \frac{u_1}{\sqrt{d}}, \ldots, \frac{u_m}{\sqrt{d}}\right).\ee
Next result gives a description of the properties of the rescaled Kostlan polynomial, in particular its convergence in law as a random element of the space of smooth functions, space which, from now, on we will always assume to be endowed with the weak Whitney's topology as in \cite{DTGRF1}.

\begin{thm}[Properties of the rescaled Kostlan]\label{thm:Kostlan}
Let $X_d:\R^m\to \R^{k}$ be the Gaussian random field defined in \eqref{eq:rescaledKostlan}. 
\begin{enumerate}[1.]
\item \emph{(The limit)} Given a family of independent gaussian random vectors $\xi_\beta \sim N\left(0,\frac{1}{\beta !}\mathbbm{1}_k\right)$, the series 
\be 
X_\infty(u)=\sum_{\beta\in \N^{m}} \xi_\beta u^\beta  ,
\ee
is almost surely convergent in $\Cr \infty {\R^m}k$ to the Gaussian Random Field\footnote{$X_\infty$ is indeed a random analytic function, commonly known as the Bargmann-Fock ensemble.} $X_\infty\in \g \infty{\R^m}k$.
\item \emph{(Convergence)} $X_d\nrw X_\infty$ in $\g \infty {\R^m}k$, that is:
\be 
\lim_{d\to +\infty}\E\{F(X_d)\}=\E\{F(X_\infty)\}
\ee 
for any bounded and continuous function $F\colon \Cr \infty{\R^m}k\to \R$. Equivalently, we have
\be 
\label{eq:inequKostlan}\P\{X_\infty \in \emph{int}(A)\}\le \liminf_{d\to +\infty}\P\{X_d\in A\}\le \limsup_{d\to +\infty}\P\{X_d\in A\}\le \P\{X_\infty\in \overline{A}\}
\ee
for any Borel subset $A\subset \Cr \infty {\R^m}k$.
\item \emph{(Nondegeneracy of the limit)} The support of $X_\infty$ is the whole $\Cr \infty {\R^m}k$. In other words, for any non empty open set $U\subset \Cr\infty{\R^m}k$ we have that $\P\{X_\infty \in U\}>0$.
\item \emph{(Probabilistic Transversality)} For $d\ge r$ and $d=\infty$, we have $\mathrm{supp}(j^r_pX_d)=J_p^r(\R^m,\R^k)$ for every $p\in \R^m$ and consequently for every submanifold $W\subset J^r(\R^m,\R^k)$, we have 
\be 
\P\{j^rX_d\transv W\}=1.
\ee
\item \emph{(Existence of limit probability)} Let $V\subset J^{r}(\R^m, \R^k)$ be an open set whose boundary is a (possibly stratified) submanifold\footnote{For example $V$ could be a semialgebraic set}. Then
\be 
\lim_{d\to +\infty}\P\{j^r_p X_d\in V,\ \forall p\in \R^m\}= \P\{j^r_pX_\infty(\R^m) \in V,\ \forall p\in \R^m\}.
\ee
In other words, we have equality in \eqref{eq:inequKostlan} for sets of the form $U=\{f\colon j^rf\in V\}$.
\item \emph{(Kac-Rice densities)} Let $W\subset J^r(\R^m,\R^k)$ be a semialgebraic subset of codimension $m$, such that\footnote{In this paper the symbol $\transv$ stands for ``it is transverse to''.} $W\transv J_p^r(\R^m,\R^k)$ for all $p\in M$ (i.e. $W$ is transverse to fibers of the projection of the jet space). Then for all $d\ge r$ and for $d=+\infty$ there exists a locally bounded function $\rho_d^W\in L^\infty_{loc} (\R^m)$ such that\footnote{A formula for $\rho_d^W$ is presented in \cite{DTGRF1}, as a generalization of the classical Kac-Rice formula.} 
\be 
\E\#\{u\in A\colon j^r_uX_d\in W\}=\int_A\rho^W_d,
\ee
for any Borel subset $A\subset \R^m$. Moreover $\rho^W_d  \to \rho^W_\infty$ in $L^\infty_{loc}$.
\item \label{thm:Ebetti}\emph{(Limit of Betti numbers)} Let $W\subset J^r(\R^m,\R^k)$ be any closed semialgebraic subset transverse to fibers. Then:
\be \label{eq:localEbetti}
\lim_{d\to +\infty}\E\left\{b_i\left((j^{r}X_d)^{-1}(W)\cap\mathbb{D}^m\right)\right\}=\E\left\{b_i\left((j^{r}X_\infty)^{-1}(W)\cap \mathbb{D}^m\right)\right\},
\ee
where $b_i(Z)=\dim H_i(Z,\R)$. Moreover, if the codimension of $W$ is $l\ge 1$, then the r.h.s. in equation \eqref{eq:localEbetti} is strictly positive for all $i=0,\dots, m-l$.
\end{enumerate}
\end{thm}
\begin{proof}The proof uses a combination of results from \cite{DTGRF1}.
\begin{enumerate}[$(1)$]
\item Let $S_d=\sum_{|\beta|\le d}\xi_\beta u^\beta\in \g \infty Mk$. The covariance function of $S_d$ converges in Whitney's weak topology: \be K_{S_d}(u,v)=\sum_{|\beta|\le d}\frac{u^\beta v^\beta}{\beta !}\mathbbm{1}_k\xrightarrow{\mathcal{C}^\infty} \exp(u^Tv)\mathbbm{1}_k.\ee 
It follows by \cite[Theorem 3]{DTGRF1} that $S_d$ converges in $\g \infty Mk$, moreover since all the terms in the series are independent we can conclude with the Ito-Nisio \footnote{It may not be trivial to apply the standard Ito-Nisio theorem, which actually regards convergence of series in a Banach space. See Theorem $36$ of \cite{DTGRF1} for a statement that is directly applicabile to our situation} Theorem \cite{ItoNisio} that indeed the convergence holds almost surely.
\item By  \cite[Theorem 3]{DTGRF1} it follows from convergence of the covariance functions:
\be 
K_{X_d}(u,v)=\left(1+\frac{u^Tv}{d}\right)^d\mathbbm{1}_k \quad \xrightarrow{\mathcal{C^\infty}}\quad  K_{X_\infty}(u,v)=\exp(u^Tv)\mathbbm{1}_k
\ee
%(see Remark \ref{exponentialcor} in Appendix A).
\item The support of $X_\infty$ contains the set of polynomial functions $\R[u]^k$, which is dense in $\Cr \infty {\R^m}k$, hence the thesis follows from  \cite[Theorem 6]{DTGRF1}.
\item Let $d\ge r$ or $d=+\infty$. We have that
\be 
\begin{aligned}
\spt(j_u^rX_d)&=\{j^r_uf\colon f\in \R[u]^k \text{ of degree $\le d$}\}=\\
&=\textrm{span}\{j^r_uf\colon f(v)=(v-u)^\beta \text{ with $|\beta|\le d$}\}=\\
&=\textrm{span}\{j^r_uf\colon f(v)=(v-u)^\beta \text{ with $|\beta|\le r$}\}=\\
&=J^r_u(\R^m,\R^k).
\end{aligned}
\ee
The fact that $\P\{j^rX_d\transv W\}=1$ follows  \cite[Theorem 8]{DTGRF1}.
\item Let $A=\{f\in \Cr \infty{\R^m}k\colon j^rf\in V\}$. If $f\in \de A$, then $j^rf\in \overline{V}$ and there is a point $u\in \R^m$ such that $j^r_uf\in \de V$. Let $\partial V$ be stratified as $\partial V=\coprod Z_i$ with each $Z_i$ a submanifold. If $j^rf\transv \de V$ then it means that $j^r f$ is transversal to all the $Z_i$ and there exists one of them which contains $j^r_uf$ (i.e. the jet of $f$ intersect $\partial V$). Therefore the intersection would be transversal \emph{and nonempty}, and then there exists a small Whitney-neighborhood of $f$ such that for every $g$ in this neighborhood $j^rg$ still intersects $\partial V.$ This means that there is a neighborhood of $f$ consisting of maps that are not in $A$, which means $f$ has a neighborhood contained in $A^c$. It follows that $f\notin \overline{A}$ and consequently $f\notin \partial A$, which is a contradiction. Therefore we have that
\be 
\de A\subset \{f\in \Cr \infty{R^m}k\colon f \text{ is not transverse to }\de V\}.
\ee
It follows by point $(4)$ that $\P\{X\in \de A\}=0$, so that we can conclude by points $(2)$ and $(3)$.
\item By previous points, we deduce that we can apply the results described in section $7$ of \cite{DTGRF1}.
\item This proof is postponed to Section \ref{sec:betti}.
\end{enumerate} 
\end{proof}
Given  a $\mathcal{C}^{\infty}$ Gaussian Random Field $X:\R^m\to \R^k$ , let us denote by $[X]$ the probability measure induced on $\mathcal{C}^{\infty}(\R^m, \R^k)$ and defined by:
\be [X](U)=\mathbb{P}(X\in U),\ee
for every $U$ belonging to the Borel $\sigma-$algebra relative to the weak Whitney topology, see \cite{Hirsch} for details on this topology.
Combining Theorem \ref{thm:Kostlan} with Skorohod Theorem \cite[Theorem 6.7]{Billingsley} one gets that it is possible to represent $[X_d]$ with equivalent fields $\tilde{X}_d$ such that $\tilde{X}_d\to \tilde{X}_\infty$ almost surely in $\Cr\infty{\R^m}k$. This is in fact equivalent to point $(2)$ of Theorem \ref{thm:Kostlan}. In other words there is a (not unique) choice of the gaussian coefficients of the random polynomials in $\eqref{eq:Kd}$, for which the covariances $\E\{\tilde{X}_d\tilde{X}_{d'}^T\}$ are such that the sequence converges almost surely. We leave to the reader to check that a possible choice is the following. 
Let $\{\gamma_\beta\}_{\beta\in\N^m}$ be a family of i.i.d. gaussian random vectors $\sim N(0,\mathbbm{1}_k)$ and define for all $d<\infty$
\be \label{eq:askocond}
\tilde{X}_d=\sum_{|\beta|\le d}{d\choose \beta}^\frac12 \gamma_\beta\left(\frac{u}{\sqrt{d}}\right)^{\beta}
\ee
and
\be\label{eq:askoconf}
\tilde{X}_\infty=\sum_{\beta}{\left(\frac{1}{\beta!}\right)}^\frac12 \gamma_\beta u^{\beta}.
\ee
\begin{prop}
$\tilde{X}_d\to \tilde{X}_{\infty}$ in $\Cr \infty{\R^m}k$ almost surely.
\end{prop}
However, we stress the fact that in most situations: when one is interested in the sequence of probability measures $[X_d]$, it is sufficient to know that such a sequence exists.
\subsection{Limit laws for Betti numbers and the generalized square-root law}\label{sec:betti}
\begin{figure}
\includegraphics[scale=0.15]{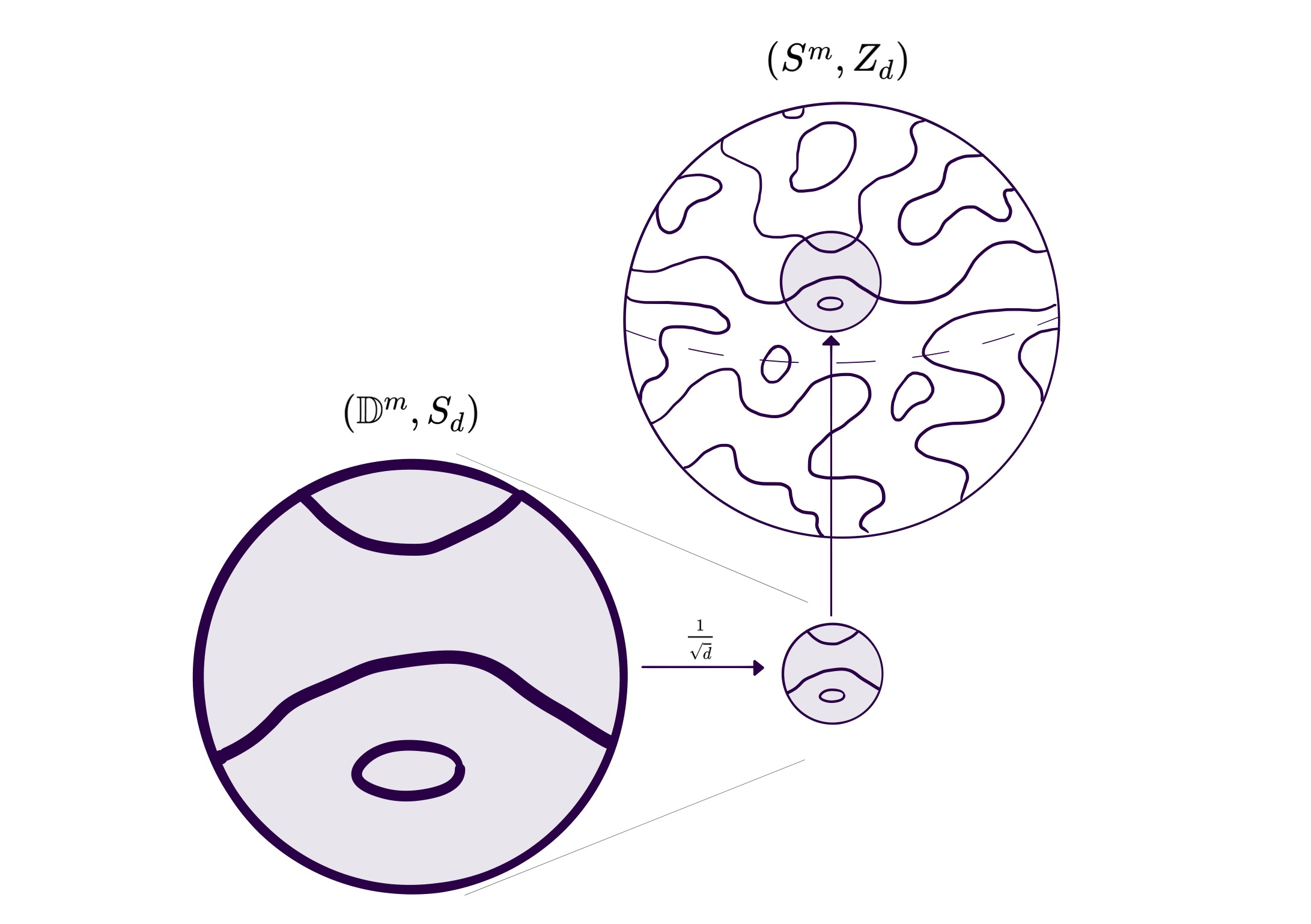}\caption{The random set $S_d=\{X_d=0\}\subset \D^m$ is a rescaled version of $Z_d\cap D(p, d^{-1/2})$, where $Z_d=\{\psi_d=0\}$.}\label{fig:rescale}
\end{figure}

Let $W_0\subset J^r(\R^m,\R^k)$ be a semialgebraic subset.
Consider the random set
\be S_d= \{p\in \mathbb{D}^m\colon j_p^rX_d\in W_0\},\ee
where $X_d\colon \R^m\to \R^k$ is the rescaled Kostlan polynomial from Theorem \ref{thm:Kostlan} (see Figure \ref{fig:rescale}).
We are now in the position of complete the proof of Theorem \ref{thm:Kostlan} by showing point (7).
Let us start by proving the following Lemma.
\begin{lemma}\label{lem:positibetti}
Let $r$ be the codimension of $W_0$ and suppose $0\le i\le m-r\le m-1$. Then \be \E\{b_i(S_\infty)\}>0 .\ee
\end{lemma}
\begin{proof}
From Corollary \ref{cor:topositive} we deduce that there exists a function $f\in \Cr \infty{\mathbb{D}^m} k$ such that $j^rf\transv W_0$ and $b_i\left((j^rf)^{-1}(W_0)\right)\neq 0$.
Since the condition on $f$ is open, there is an open neighbourhood $O$ of $f$ where  $b_i((j^rg)^{-1}(W_0))=c>0$ for all $g\in O$. Thus $\P\{b_i(S_\infty)=c\}>0$ because every open set has positive probability for $X_\infty$, by \ref{thm:Kostlan}.3. therefore $\E\{b_i(S_\infty)\}>0$. 
%%On the other hand if no such function $f$ exists, then in particular there is no polynomial function with that property, hence
%\be 
%\P\{b_i(S_d)>0\}=\P\{X_d\transv W, b_i(S_d)>0\}=\P\{X_d\in \emptyset\}=0.
%\ee
%The first identity follows from the transversality property \ref{thm:Kostlan}.4.
\end{proof}
We complete the proof of Theorem \ref{thm:Kostlan} with the next Proposition.
\begin{prop}
\be\label{eq:Ebetti}
\lim_{d\to \infty}\E\{b_i(S_d)\}=\E\{b_i(S_\infty)\}.
\ee
\end{prop}
%\be
%\includegraphics[scale=0.15]{figure5.jpg}
%\ee
%(Notice that \eqref{eq:Ebetti} doesn't follow from Example \ref{ex:Ebetti} below, since $b_i(\cdot)$ is a not bounded function.)
\begin{proof}
%Let $X_d$ and $X_\infty$ be defined as in \eqref{eq:askocond} and \eqref{eq:askoconf}, then $X_d\to X_\infty$ almost surely, as proven by Corollary \ref{cor:askocon} in Appendix A.
Let $b_i(S_d)=b_d$.
Define a random field $Y_d=(X_d,x_d)\colon \R^m\to \R^{k}\times\R$ to be the rescaled Kostlan polynomial of type $(m,k+1)$.
Consider the semialgebraic subset $W'=W\cap J^r(\mathbb{D}^m,\R^k)$ of the real algebraic smooth manifold $J^r(\R^m,\R^k)$ and observe that $S_d=(j^rX_d)^{-1}(W')$ is compact. Now Theorem \ref{thm:strat}, along with Remark \ref{rem:strat}, implies the existence of a semialgebraic submanifold $\hat{W'}\subset J^{r+1}(\R^m,\R^{k+1})$ of codimension $m$ and a constant $C$, such that 
\be 
b_d\le C\#\left\{\left(j^{r+1}(Y_d)\right)^{-1}(\hat{W'})\right\}=:N_d
\ee
whenever $j^rX_d \transv W'$ and $j^{r+1}Y_d\transv \hat{W'}$, hence almost surely, because of Theorem \ref{thm:Kostlan}.4.
Since $Y_d\nrw Y_\infty$ by \ref{thm:Kostlan}.2, we see that $[b_d,N_d]\nrw [b_\infty,N_\infty]$ and it is not restrictive to assume that $(b_i,N_d)\to (b_i,N_\infty)$ almost surely, by Skorokhod's theorem (see \cite[Theorem 6.7]{Billingsley}). Moreover $\E\{N_d\}\to \E\{N_\infty\}$ by Theorem \ref{thm:Kostlan}.6.
Now we can conclude with Fatou's Lemma as follows 
\be 
\begin{aligned}
2\E\{N_\infty\}&=\E\{\liminf_d N_d+N_\infty-|b_d-b_\infty|\}\le \\
&\le \liminf_d\E\{ N_d+N_\infty-|b_d-b_\infty|\}=\\
&=
2\E\{N_\infty\}-\limsup_d\E\{|b_d-b_\infty|\},
\end{aligned}
\ee 
so that
\be 
\limsup_d\E\{|b_d-b_\infty|\}\le 0.
\ee
%Transversality implies also that $(j^{r+1}Y_d)^{-1}(W)\subset \text{int}(\mathbb{D}^m)$ with probability one, and it is easy to see that $\hat{W'}|_{\text{int}(\mathbb{D}^m)}=\hat{W}|_{\text{int}(\mathbb{D}^m)}$. Recalling that $\hat{W}$ was intrinsic and semialgebraic (see Remark \ref{rem:strat}) we are enabled to apply the generalized Kac-Rice formula from Theorem \ref{thm:Kostlan}.6, obtaining smooth functions $\rho_d^{\hat{W}}$, such that
%\be\label{eq:ecritbetti}
% \E N_d=\int_{\mathbb{D}^m} \rho_d^{\hat{W}}du.
%\ee
%From Theorem \ref{thm:Kostlan} we know that $Y_d\nrw Y_\infty$ and that $\rho_d^{\hat{W}}\to \rho_\infty^{\hat{W}}$ uniformly on $\mathbb{D}^m$, therefore passing to the limit in equation \eqref{eq:ecritbetti} we get $\E N_d\to \E N_\infty$. Moreover, applying Lemma \ref{discretelemma}, we see that $[b(Z_d),N_d]\nrw [b(Z_\infty),N_\infty]$.
\end{proof}
%Next Proposition concludes the proof of \ref{thm:Kostlan}.7.

%\begin{thm}
%Let $W\subset J^r(S^m,\R^k)$ be a semialgebraic and intrinsic submanifold with local model $\widehat{W}$ and let $\psi_d=P_d|_{S^m}\in \g \infty{S^m}k$ as in Theorem \ref{thm:sqrlaw}. consider the random subset of $S^m$ defined as $Z_d=j^r\psi_d^{-1}(W)$. There exists constants $C_1, C_2>0$ such that for all $d$ big enough we have
%\be 
%C_1d^{\frac{m}{2}}\le \E\{b(Z_d)\}\le C_2d^{\frac{m}{2}}.
%\ee
%\end{thm}

In the sequel, with the scope of keeping a light notation, for a given $W\subset J^{r}(S^m, \R^k)$ and $\psi:S^m\to \R^k$ we will denote by $Z_d\subseteq S^m$ the set
\be Z_d=j^r\psi^{-1}(W).\ee
If $W$ is of codimension $m$, then by Theorem \ref{thm:Kostlan}, $Z_d$ is almost surely a finite set of points and the expectation of this number is given by next result.
\begin{thm}[Generalized square-root law for cardinality]\label{thm:sqrlaw}
Let $W\subset J^r(S^m,\R^k)$ be a semialgebraic intrinsic subset of codimension $m$. Then there is a constant $C_W>0$ such that:
%for any Borel subset $A\subset S^m$.
\be 
\E\{\#Z_d\}=C_W d^{\frac{m}{2}}+ O(d^{\frac{m}{2}-1}).
\ee
Moreover, the value of  $C_W$ can be computed as follows. Let $Y_\infty=e^{-\frac{|u|^2}{2}}X_\infty\in \g \infty{\mathbb{D}^m}k$ and let $W_0\subset J^r(\mathbb{D}^m,\R^k)$ be the local model for $W$. Then
\be 
C_W=m\frac{\vol (S^m)}{\vol (S^{m-1})} \E\#\{u\in\mathbb{D}^m\colon j^r_uY_{\infty}\in W_0\}.
\ee
\end{thm}

In order to prove Theorem \ref{thm:sqrlaw}, we will need a preliminary Lemma, which ensures that we will be in the position of using the generalized Kac-Rice formula of point (6) from Theorem \ref{thm:Kostlan}.
\begin{lemma}\label{lem:intrilinstab}
If $W\subset J^r(M,\R^k)$ is intrinsic, then $W$ is transverse to fibers.
\end{lemma}
\begin{proof}
Since the result is local it is sufficient to prove it in the case when $M=\R^m$. 
In this case we have a natural identification (see \cite[Chapter 2, Section 4]{Hirsch})

For any point $u\in \R^m$ we consider the embedding $ i_u \colon \mathbb{D}^m\to\R^m$ obtained as the isometric inclusion in the disk with center $u$ and let $\tau_u\colon \R^m\to \R^m$ be the translation map $x\mapsto u+x$.
Let $u,v\in \R^m$ be two points with distance smaller than $1$, he fact that the submanifold $W$ is intrinsic implies that $j^r_vf\in W$ if and only if $(j^ri_u)^*(j^r_vf)\in W_0$, where $W_0\subset J^r(\mathbb{D}^m,\R^k)$ is the model for $W$. From this we deduce that also the jet $j^r_u(f\circ \tau_{v-u})$ is in $W$, since:
\be \begin{aligned}
\left(j^ri_v\right)^*\left(j^r_{v}f\right)&=
j^r(\tau_{v-u}\circ i_u)^*(j^r_vf)
\\
&= 
\left(j^ri_u\right)^*\left(j^r\left(\tau_{v-u}\right)^*\left(j^r_{\tau_{v-u}(u)}f\right)\right)
\\
&= 
\left(j^ri_u\right)^*\left(j^r_u\left(f\circ\tau_{v-u}\right)\right).
\end{aligned}
\ee
By interchanging the role of $u$ and $v$, we conclude that $j^r_u(f\circ\tau_{v-u})\in W$ if and only if $j^r_vf\in W$. Notice that such statement is thus true for any couple of points $u,v\in\R^m$, regardless of their distance.

We thus claim that $T(W)$ is of the form $\R^m\times \bar{W}$, under the natural identification (see \cite[Sec. 2.4]{Hirsch}):
\be 
T\colon J^r(\R^m,\R^k)\cong \R^m\times J^r_0(\R^m,\R^k), \qquad j^r_uf\mapsto (u,j^r_0(f\circ \tau_{u})).
\ee
To see this, observe that if $(v,j^r_0g)\in T(W)$, hence $(v,j^r_0g)= T(j^r_vf)$ for a jet $j^r_vf\in W$ such that $ 
g=f\circ \tau_{v}$,
then $(u,j^r_0g)=T\left(j^r_u(f\circ \tau_{v-u})\right)\in T(W)$.
%so that $\left(u,j^r_0f\right)\in W$. It follows that, calling $\{0\}\times \bar{W}= W\cap J^r_0(\R^m,\R^k)$,
%\be 
%W=\{(u,j_0^rf)\colon j^r_0f \in W_0\}=\R^m\times \bar{W},
%\ee
%which is clearly transverse to each fiber $\{u\}\times J^r_0(\R^m,\R^k)$.
\end{proof}
%In the following theorem we are going to resume the notation of Corollary \ref{thm:kacarisoboss}.
%\begin{thm}[The generalized Square-Root Law]\label{thm:sqrlaw}
%Let $W\subset J^r(S^m,\R^k)$ be a semialgebraic and intrinsic submanifold with local model $\widehat{W}$ (see Definition \ref{def:intrinsic}) of codimension $m$. Let $\psi_d=P_d|_{S^m}\in \g \infty{S^m}k$ be the restriction of the Kostlan polynomial to the sphere. Then there is a constant $c_{W, d}\in \R$ such that $\rho_{W, d}=c_{W,d} \vol_{S^m} \in \mathscr{D}(S^m)$. As $d\to+\infty$ we have \be 
%\E\#_{j^r\psi_d\in W}(S^m)={c_d^W}{\vol(S^m)}= c_Wd^\frac{m}{2}+O(d^{\frac{m}{2}-1}).
%\ee
%The constant $c_W$ is determined by the following formula involving the random field
%\be Y_\infty=e^{-\frac{|u|^2}{2}}X_\infty\in \g \infty{\mathbb{D}^m}k,\ee 
%where $X_\infty$ is the rescaled kostlan limit defined in Theorem \ref{thm:Kostlan},
%\be
%c_W=m\frac{\vol(S^m)}{\vol(S^{m-1})}\E\#_{j^rY_\infty\in \widehat{W}}\left(\mathbb{D}^m\right).
%\ee
%\end{thm}
\begin{figure}
\includegraphics[scale=0.11]{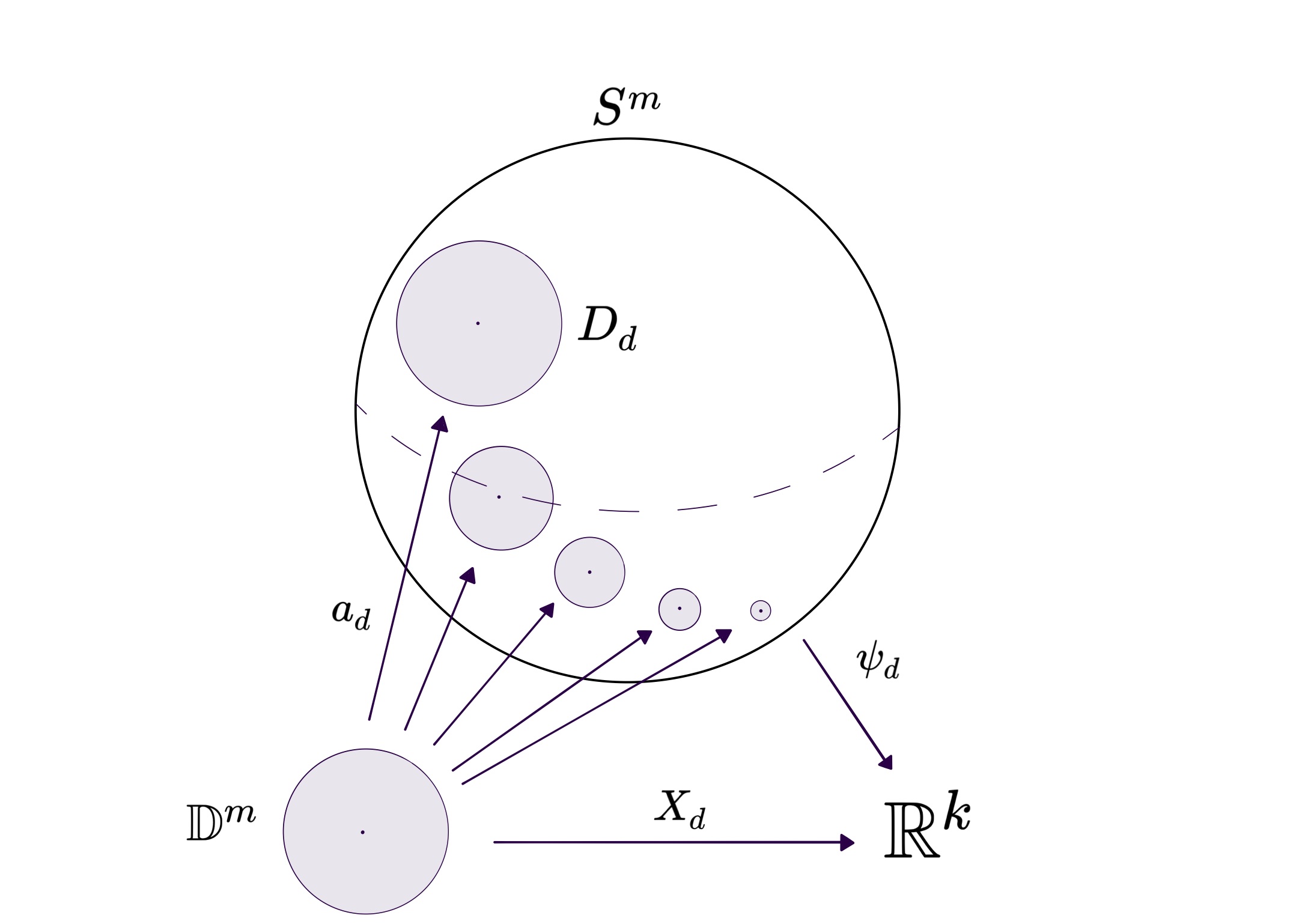}\caption{A family of shrinking embedding of the unit disk.}\label{fig:emb}
\end{figure}
The reason why we consider intrinsic submanifold is to be able to easily pass to the rescaled Kostlan polynomial $X_d\in \g \infty{\mathbb{D}^m}k$ by composing $\psi_d$ with the embedding of the disk $a_d^R$ defined by: 
\be\label{eq:diskemb}
a_d^R\colon \mathbb{D}^m\hookrightarrow S^m,\quad 
u\mapsto \frac{R\begin{pmatrix}
1 \\ \frac{u}{\sqrt{d}}
\end{pmatrix}}{\sqrt{\left(1+\frac{|u|^2}{d}\right)}}
\ee
for any $R\in O(m+1)$ (see Figure \ref{fig:emb}).
\begin{proof}[Proof of Theorem \ref{thm:sqrlaw}]
Let us consider the set function $\mu_d\colon \mathcal{B}(S^m)\mapsto \R$ such that $A\mapsto \E\{\#(j^rX_d)^{-1}(W)\cap A\}$. It is explained in \cite{DTGRF1} that $\mu_d$ is a Radon measure on $S^m$. Because of the invariance under rotation of $P_d$, by Haar's theorem $\mu$ needs to be proportional to the volume measure. Therefore for any Borel subset $A\subset S^m$ we have $\E\{\#Z_d\}=\mu_d(S^m)=\mu_d(A)\vol(A)^{-1}\vol(S^m)$. 
Define $Y_d\in \g \infty {\mathbb{D}^m}k$ as
\be\label{eq:Y}
Y_d= \left(1+\frac{|u|^2}{d}\right)^{-\frac{d}2}X_d.
\ee
Observe that $Y_d\nrw Y_\infty=\exp(-\frac{|u|^2}{2})X_\infty$ and that $Y_d$ is equivalent to the GRF $\psi_d\circ a^R_d$ for any $R\in O(m+1)$.

Now let $W_0\subset J^r(\mathbb{D}^m,\R^k)$ be the (semialgebraic) model of $W$. By the same proof of point (7) from Theorem \ref{thm:Kostlan}, adapted to $Y_d$, there is a convergent sequence of functions $\rho_d\to \rho_{+\infty}\in L^1(\mathbb{D}^m)$ such that 
\be 
\E\{\#(j^rY_d)^{-1}(W_0)\}=\int_{\mathbb{D}^m}\rho_d\to \int_{\mathbb{D}^m}\rho_\infty= \E\{\#(j^rY_\infty)^{-1}(W_0)\}.
\ee
In conclusion we have for $A=a^R_d(\mathbb{D}^m)$, as $d\to +\infty$
\be 
\begin{aligned}
\E\{\#Z_d\}&=\mu_d(A)\vol(A)^{-1}\vol(S^m)\\
&=\E\{\#(j^rY_d)^{-1}(j^r\f^*(W))\}\vol(A)^{-1}\vol(S^m)\\
&=\E\{\#(j^rY_d)^{-1}(W_0)\}\left(\frac{\int_0^{\pi}|\sin\theta|^{m-1}d\theta}{\int_0^{\arctan\left(d^{-\frac12}\right)}|\sin\theta|^{m-1}d\theta} \right)\\
&=\E\{\#(j^rY_\infty)^{-1}(W_0)\}m\frac{\vol(S^m)}{\vol(S^{m-1})}d^{\frac{m}{2}}+O(d^{\frac{m}{2}-1}).
\end{aligned}
\ee

\end{proof}
Building on the previous results, we can now prove the general case for Betti numbers of a random singualrity.
\begin{thm}[Generalized square-root law for Betti numbers]\label{thm:bettiorder}
Let $W\subset J^r(S^m,\R^k)$ be a closed semialgebraic intrinsic (as defined in Definition \ref{def:intrinsic}) of codimension $1\le l\le m$. Then there are constants $b_W, B_W > 0$ depending only on $W$ such that
\be \label{eq:ineqbetti}
b_W d^{\frac{m}{2}}\le\E\{b_i(Z_d)\}\le B_W d^{\frac{m}{2}}\quad \forall i=0,\dots, m-l
\ee
and $\E\{b_i(Z_d)\}=0$ for all other $i$.
%\be 
%\E\{b_i(Z_d)\}\approx_{d\to +\infty}d^{\frac{m}{2}}.
%\ee
%\be
%b_W<\frac{\E\{b_i(Z_d)\}}{d^{\frac{m}{2}}}=_{d\to +\infty}O(1).
%\ee
%\be 
%0<\liminf_{d\to+\infty}\frac{\E\{b_i(Z_d)\}}{d^{\frac{m}{2}}}\le\limsup_{d\to+\infty}\frac{\E\{b_i(Z_d)\}}{d^{\frac{m}{2}}}<+\infty
%\ee
\end{thm}

\begin{proof}The proof is divided in two parts, first we prove the upper bound, using the square-root law from Theorem \ref{thm:sqrlaw}, then the we use Theorem \ref{thm:Ebetti} to deduce the lower bound. The globalization step for the lower bound is a generalization of the so-called ``barrier method'' from \cite{NazarovSodin2, GaWe1}.

1. Assume $W$ is smooth with codimension $s$. Let us consider 
\be P^{m, k+1}_d|_{S^m}=\Psi_d=(\psi_d,\psi_d^1)\in \g \infty {S^m}{k+1}
\ee
and Let $\hat{W}\subset J^{r+1}(S^m,\R^{k+1})$ be the intrinsic semialgebraic submanifold coming from Theorem \ref{thm:strat} and Remark \ref{rem:strat}. Thus, using Theorems \ref{thm:strat} and \ref{thm:sqrlaw}, we get
\be 
\E\{b_i(Z_d)\}\le N_W\E\#\{(j^{r+1}\Psi_d)^{-1}(\hat{W})\}\le N_WC_{\hat{W}}d^{\frac{m}{2}}.
\ee

2. Consider the embeddings of the $m$ dimensional disk $a_d^R\colon \mathbb{D}^m\hookrightarrow S^m$ defined in \eqref{eq:diskemb}. For any fixed $d\in\N$, choose a finite subset $F_d\subset O(m+1)$ such that the images of the corresponding embeddings $\{a_d^R(\mathbb{D}^m)\}_{R\in F_d}$ are disjoint. Denoting by $Z_d^R$ the union of all connected components of $Z_d$ that are entirely contained in $a_d^R(\mathbb{D}^m)$, we have
\be 
b_i(Z_d)\ge \sum_{R\in F_d}b_i(Z_d^R).
\ee
Let $W_0\subset J^r(\mathbb{D}^m,\R^k)$ be the model of $W$ as an intrinsic submanifold, it is closed and semialgebraic. By Definition \ref{def:intrinsic}, we have
\be\label{eq:pezzi}
(a_d^R)^{-1}\left((j^r\psi_d)^{-1}(W)\right)= \left(j^r(\psi_d\circ a_d^R)\right)^{-1}(W_0)\subset \mathbb{D}^m.
\ee
Recall that for any $R\in O(m+1)$, the GRF $\psi_d\circ a^R_d$ is equivalent to $Y_d\in \g \infty{\mathbb{D}^m}k$ defined in \ref{eq:Y}, hence taking expectation in Equation \eqref{eq:pezzi} we find
\be 
\E \{b_i(Z_d)\}\ge \#(F_d)\E \{b_i(S_d)\},
\ee
where $S_d=\left(j^r(Y_d)\right)^{-1}(W_0)$.
 is easy to see (repeating the same proof) that Theorem \ref{thm:Kostlan}.7 holds also for the sequence $Y_d\nrw Y_\infty$, so that $\E\{S_d\}\to \E\{S_\infty\}$. We can assume that $\E\{S_\infty\}>0$, because of Lemma \ref{lem:positibetti}, thus for big enough $d$, the numbers $\E\{b_i(S_d)\}$ are bounded below by a constant $C>0$.
Now it remains to estimate the number $\#(F_d)$. Notice that $a_d^R(\mathbb{D}^m)$ is a ball in $S^m$ of a certain radius $\e_d$, hence it is possible to choose $F_d$ to have at least $N_m \e_d^{-1}$ elements, for some dimensional constant $N_m>0$ depending only on $m$. We conclude by observing that 
\be 
\e_d\approx d^{-\frac{m}{2}}.
\ee
\end{proof}
\section*{Appendix 1: Examples of applications of Theorem \ref{thm:Kostlan}}\label{sec:examples}
\begin{example}[Zero sets of random polynomials]\label{ex:welsch}Consider the zero set $Z_d\subset \R\P^m$ of a random Kostlan polynomial $P_d=P_{d}^{m+1,1}$. Recently Gayet and Welschinger \cite{GaWe1} have proved that given a compact hypersurface $Y\subset \R^{m}$ there exists a positive constant $c=c(\R^m, Y)>0$ and $d_0=d_0(\R^m,Y)\in \N$ such that for every point $x\in \R\P^m$ and every large enough degree $d\ge d_0$, denoting by $B_d$ any open ball of radius $d^{-1/2}$ in $\R\P^m$, we have:
\be \label{eq:isotopic}\left(B_d, B_d\cap Z_d\right)\cong(\R^m, Y) \ee
(i.e. the two pairs are diffeomorphic) with probability larger than $c$. This result follows from Theorem \ref{thm:Kostlan} as follows. Let $\mathbb{D}^m\subset \R^m$ be the unit disk, and let $U\subset \mathcal{C}^{\infty}(\mathbb{D}^m, \R)$ be the open set consisting of functions $g:\mathbb{D}^m\to \R$ whose zero set is regular (an open $\mathcal{C}^1$ condition satisfied almost surely by $X_d$, because of point (4)), entirely contained in the interior of $\mathbb{D}^m$ (an open $\mathcal{C}^0$ condition) and such that, denoting by $\B\subset \R^m$ the standard unit open ball,  the first two conditions hold and $(\mathbb{B}, \mathbb{B}\cap\{g=0\})$ is diffeomorphic to $(\R^m, Y)$ (this is another open $\mathcal{C}^1$ condition).
Observe that, using the notation above:
\be \left(B_d, B_d\cap Z_d\right)\sim (\mathbb{B}, \mathbb{B}\cap\{X_d=0\}) \ee
(this is simply because $X_d(u)=P_d(1, ud^{-1/2})$). Consequently point (5) of Theorem \ref{thm:Kostlan} implies that:
\begin{align}
\lim_{d\to +\infty}\P\{\eqref{eq:isotopic}\}&=\lim_{d\to \infty} \P\left\{(\mathbb{B}, \mathbb{B}\cap\{X_d=0\})\sim (\R^m, Y)\right\}\\
&=\lim_{d\to \infty} \P\left\{X_d\in U\right\}\\
&=\P\left\{X_\infty \in U \right\}>0.
\end{align}
We stress that, as an extra consequence of Theorem \ref{thm:Kostlan}, compared to \cite{GaWe1} what we get is the existence of the limit of the probability of seeing a given diffeomorphism type.
\end{example}
\begin{example}[Discrete properties of random maps]\label{sec:discrete}
Let $[X_d]\nrw [X_\infty]$ be a converging family of gaussian random fields. In this example we introduce a useful tool for studying the asymptotic probability induced by $X_d$ on discrete sets as $d\to \infty$.
The key example that we have in mind is the case when we consider a codimension-one ``discriminant'' $\Sigma \subset \mathcal{C}^\infty(S^m, \R^k)$  which partitions the set of functions into many connected open sets. For instance $\Sigma$ could be the set of maps for which zero is not a regular value: the complement of $\Sigma$ consists of countably many open connected sets, each one of which corresponds to a rigid isotopy class of embedding of a smooth codimension-$k$ submanifold $Z\subset S^m$. The following Lemma gives a simple technical tool for dealing with these situations.

\begin{lemma}\label{discretelemma}
Let $E$ be a metric space and let $[X_d], [X_\infty]$ be a random fields such that $[X_d]\nrw [X_\infty]$. Let also $Z$ be a discrete space and $\nu \colon U\subset E\to Z$ be a continuous function defined on an open subset $U\subset E$ such that\footnote{Of course, $E\smallsetminus U=\Sigma$ is what we called ``discriminant'' in the previous discussion. Note that we do not require that $\P\{X_d\in U\}=1$, however it will follow that  $\lim_d \P\{X_d\in U\}=1$.} $\P\{X_\infty\in U\}=1$. Then, for any $A\subset Z$ we have:
\be 
\exists \lim_{d\to \infty}\P\left\{X_d\in U,\ \nu(X_d)\in A\right\}=\P\left\{\nu(X_\infty)\in A\right\}.
\ee
\end{lemma}
\begin{proof}
Since $\nu^{-1}(A)$ is closed and open by continuity of $\nu$, it follows that $\de \nu^{-1}(A)\subset E\backslash U$. Therefore $\P\{X_\infty\in \de \nu^{-1}(A)\}=0$ and by Portmanteau's Theorem \cite[Theorem 2.1]{Billingsley}, we conclude that 
\be \label{discreteq}
\P\{X_d\in \nu^{-1}(A)\}\xrightarrow[d\to\infty]{}\P\{X_\infty\in \nu^{-1}(A)\}, \quad \ \forall \ A\subset Z.
\ee
\end{proof}
Equation \eqref{discreteq}, in the case of a discrete topological space such as $Z$, is equivalent to narrow convergence $\nu(X_d)\nrw \nu(X)$, by Portmanteau's Theorem, because $\de A=\emptyset $ for all subsets $A\subset Z$. Note also that to prove narrow convergence of a sequence of measures on $Z$, it is sufficient to show \eqref{discreteq} for all $A=\{z\}$, indeed in that case the inequality 
\be 
\liminf_{d\to \infty}\P\{\nu_d\in A\}=\liminf_{d\to \infty}\sum_{z\in A}\P\{\nu_d=z\}\ge \sum_{z\in A}\P\{\nu=z\}=\P\{\nu\in A\}
\ee
follows automatically from Fatou's lemma.

Following Sarnak and Wigman \cite{SarnakWigman}, let us consider one simple application of this Lemma. Let $H_{m-1}$ be the set of diffeomorphism classes of smooth closed connected hypersurfaces of $\R^{m}$. Consider $U=\{ f\in \Cr \infty {\mathbb{D}^m}{}\,\colon\, f\transv 0\}$ and let $\nu(f)$ be the number of connected components of $f^{-1}(0)$ entirely contained in the interior of  $\mathbb{D}^m$. For $h\in H_{m-1}$ let $\nu_{h}(f)$ be the number of those components which are diffeomorphic to $h\subset \R^{m}$. In the spirit of \cite{SarnakWigman}, we define the probability measure $\mu(f)\in \mathscr{P}(H_{m-1})$ as 
\be 
\mu(f)=\frac{1}{\nu(f)}\sum_{h\in H_{m-1}}\nu_h(f)\delta_h.
\ee 
Let us consider now the rescaled Kostlan polynomial $X_d:\mathbb{D}^m\to \R$ as in Theorem \ref{thm:Kostlan}. The diffeomorphism type of each internal component of $f^{-1}(0)$ remains the same after small perturbations of $f$ inside $U$, hence $\mu\colon U\to \mathscr{P}(H_{m-1})$ is a locally constant map, therefore by Lemma \ref{discretelemma} we obtain that for any subset $A\subset\mathscr{P}(H_{m-1})$,
\be 
\exists \lim_{d\to \infty}\P\{X_d\in U \text{ and }\mu(X_d)\in A\}=\P\{\mu(X_\infty)\in A\}. 
\ee
Moreover since in this case $X_d\in U$ with $\P=1$, for all $d\in \N$ and the support of $X_\infty$ is the whole $\Cr\infty {\mathbb{D}^m}{}$, we have
\be 
\exists \lim_{d\to \infty}\P\{\mu(X_d)\in A\}=\P\{\mu(X_\infty)\in A\}>0. 
\ee

%The measure $\mu_d=\mu(X_d)$ is a random variable with values in the set $\mathscr{P}(H_{m-1})$ considered with the discrete topology. %the metric space $ \mathscr{P}(H_{m-1})$ endowed with the Total Variation distance.
%Using Theorem \ref{thm:Kostlan} one can prove (see Section \ref{sec:discrete}) that the sequence $\{\mu_d\}_{d\in \N}$ converges to a limit random probability measure $\mu_\infty$ whose support is the whole $H_{m-1}$. Equivalently, for any fixed probability measure $\mu\in\mathscr{P}(H_{m-1})$, we have
%\be 
%\lim_{d\to\infty}\P\{\mu_d=\mu\}=\P\{\mu_\infty=\mu\}>0.
%\ee
\end{example}
%\begin{example}[Limit for expected Betti numbers]
%Let us consider the random closed manifold $Z_d\subset \mathbb{D}^m$ obtained as the union of all connected components of $ \{X_d=0\}$ entirely contained in the interior of $\mathbb{D}^m$, where $X_d\colon \mathbb{D}^m\to \R$ is again the rescaled Kostlan polynomial from Theorem \ref{thm:Kostlan} (restricted to the unit disk $\D\subset \R^m$). Let $b(Z_d)\in \N^{m+1}$ be the vector  of all Betti numbers of $Z_d$. A consequence of Theorem \ref{thm:Kostlan} is that $b(Z_d)\nrw b(Z_\infty)$ as probability measures on $\N^{m+1}$, in particular for any bounded function $\f\colon \N^{m+1}\to \R$ we have 
%\be\label{eq:bettinarrow}
%\lim_{d\to \infty}\E\left\{\f\left(b(Z_d)\right)\right\}=\E\left\{\f\left(b(Z_\infty)\right)\right\}.
%\ee
%We are going to prove, using a generalization of Kac-Rice formula (see Section \ref{se:Kac-Rice}) that the assumption ``$\f$ bounded '' can be removed and, as a consequence of Theorem \ref{thm:Ebetti}), we have indeed:
%\be 
%\exists\lim_{d\to \infty}\E\{b(Z_d)\}=\E\{b(Z_\infty)\}.
%\ee
%\end{example}
\begin{example}[Random rational maps] \label{ex:rational}
The Kostlan polynomial $P_d^{m, k+1}$ can be used to define random rational maps. In fact, writing $P_{d}^{m, k+1}=(p_0, \ldots, p_k)$, then one can consider the map $\varphi_{d}^{m,k}:\RP^m\dashrightarrow\RP^k$ defined by:
\be\label{eq:rrm} \varphi_d^{m,k}([x_0, \ldots, x_m])=[p_0(x), \ldots, p_m(x)].\ee
(When $m>k$, with positive probability, this map might not be defined on the whole $\RP^m$; when $m\le k$ with probability one we have that the list $(p_0, \ldots, p_k)$ has no common zeroes, and we get a well defined map $\varphi_{d}^{m,k}:\RP^m\to \RP^k.$) Given a point $x\in \RP^m$ and a small disk $D_d=D(x, d^{-1/2})$ centered at this point, the behavior of $\varphi_{d}^{m,k}|_{D_d}$ is captured by the random field $X_d$ defined in \eqref{eq:rescaledKostlan}: one can therefore apply Theorem \ref{thm:Kostlan} and deduce, asymptotic local properties of this map. 

For example, when $m\le k$ for any given embedding of the unit disk $q:\mathbb{D}^m\hookrightarrow \RP^k$ and for every neighborhood $U$ of $q(\partial \mathbb{D}^m)$ there exists a positive constant $c=c(q)>0$ such that for big enough degree $d$ and with probability larger than $c$ the map 
\be X_d=\varphi_{d}^{m,k}\circ a_d:\mathbb{D}^m\to \RP^k\ee (defined by composing $\varphi$ with the rescaling diffeomorphism $a_d:\mathbb{D}^m\to D_d$) is isotopic to $q$ thorugh an isotopy $\{q_t:\mathbb{D}^m\to \RP^k\}_{t\in I}$ such that $q_t(\partial \mathbb{D}^m)\subset U$ for all $t\in I$.

The random map $\f_d^{m,k}$ is strictly related to the random map $\psi^{m,k}_d\colon S^m\to \R^k$:
\be 
\psi^{m,k}_d(x)=P^{m,k}_d(x),
\ee
which is an easier object to work with. For example the random algebraic variety $\{\f_d=0\}$ is the quotient of $\{\psi_d=0\}$ modulo the antipodal map. If we denote by $D_d$ any sequence of disks of radius $d^{-\frac{1}{2}}$ in the sphere, then $\psi_d|_{D_d}\approx X_d$, so that we can understand the local asymptotic behaviour of $\psi_d$ using Theorem \ref{thm:Kostlan} (see Figure \ref{fig:rescale}).
For instance, from point $(7)$ it follows that 
\be 
\E\left\{b_i\left(\left\{\psi_d=0\right\}\cap D_d\right)\right\}\to \E\left\{b_i\left(\left\{X_\infty=0\right\}\cap \mathbb{D}^m\right)\right\}.
\ee
\end{example}
%\begin{example}[Singularities of random maps] An interesting example (related to the previous ones) is the case of random planar maps. Let \be \varphi_{d}^{2,2}:\RP^2\to \RP^2\ee
% be a random rational map  of degree $d$ as defined in \eqref{eq:rrm}. 
%For a generic map $g:\mathbb{D}^2\to \R^2$ only three type of singularities can appear at the origin (see \cite{Whitney}): (1) a regular point of $g$; (2)  a fold point; (3) a simple cusp. Moreover this singularities are stable, meaning that they persist on a small disk after a small perturbation of the function $g$. As a consequence, using Theorem \ref{thm:Kostlan}, one can show that for every $x\in \RP^2$ each of these singularities has a positive probability of appearing in the disk $D_d=D(x, d^{-1/2})$ (i.e. as a singularity of the map $\varphi_{d}^{2,2}|_{D_d}$). (See Section \ref{sec:discrete} for a more precise statement.)
%\label{ex:planar}
%\end{example}

\begin{example}[Random knots]
\label{ex:knots}
Kostlan polynomials offer different possible ways to define a ``random knot''. The first is by considering a random rational map:
\be \varphi_{d}^{1,3}:\RP^1\to \RP^3,\ee
to which the discussion from Example \ref{ex:rational} applies.
(Observe that this discussion has to do with the \emph{local} structure of the knot.)

Another interesting example of random knots, with a more global flavour, can be obtained as follows. Take the random Kostlan map $X_d: \R^2\to \R^3$ (as in \eqref{eq:rescaledKostlan} with $m=2$ and $k=3$) and restrict it to $S^1=\partial \mathbb{D}^m$ to define a random knot:
\be k_d=X_{d}|_{\partial \mathbb{D}^m}:S^1\to \R^3.\ee
The difference between this model and the previous one is that this is global, in the sense that as $d\to \infty$ we get a limit global model $k_\infty=X_\infty|_{\partial D}:S^1\to \R^3$.  
What is interesting for this model is that the Delbruck--Frisch--Wasserman conjecture \cite{delbruck, frischwasserman}, that a typical random knot is non-trivial, does not hold: in fact $k_\infty$ charges every knot (included the unknot) with positive probability.
\begin{prop}The random map:
\be k_d=X_d|_{\partial \D^2}:S^1\to \R^3.\ee
is almost surely a topological embedding (i.e. a knot).
Similarly, the random rational map $\varphi_{d}^{1,3}:\RP^1\to \RP^3$ is almost surely an embedding.
\end{prop}
\begin{proof}We prove the statement for $k_d$, the case of $\varphi_d^{1, 3}$ is similar.
Since $S^1$ is compact, it is enough to prove that $k_d$ is injective with probability one.

Let $F_d=\R[x_0, x_1, x_2]_{(d)}^3$ be the space of triples of homogeeous polynomials of degree $d$ in $3$ variables. Recall that $k_d=X_d|_{\partial \mathbb{D}^2}$, where, if $P\in F_d$, we have set:
\be X_d(u)=P\left(1, \frac{u}{\sqrt{d}}\right),\quad u=(u_1, u_2)\in \R^2.\ee
Let now $S^1=\partial \mathbb{D}^2\subset \R^2$ and $\phi:\left((S^1\times S^1)\backslash \Delta\right)\times F_d\to \R^3$ be the map defined by
\be \phi(x,y, P)=P\left(1, \frac{x}{\sqrt{d}}\right)-P\left(1, \frac{y}{\sqrt{d}}\right).\ee
Observe that $\phi\pitchfork \{0\}.$ By the parametric transversality theorem we conclude that $\phi$ is almost surely transversal to $W=\{0\}$. This imples that, with probability one, the set 
\be \{x\neq y\in S^1\times S^1\,|\, k_d(x)=k_d(y)\}\ee
is a codimension-three submanifold of $S^1\times S^1$ hence it is empty, so that $k_d$ is injective. 
\end{proof}

Theorem \ref{thm:Kostlan} implies now that the random variable $k_d\in C^\infty(S^1, \R^3)$ converges narrowly to $k_\infty\in C^{\infty}(S^1, \R^3)$, which is the restriction to $S^1=\partial \mathbb{D}^2$ of $X_\infty.$ Note that, since the support of $X_\infty$ is all $C^\infty(\mathbb{D}^2, \R^3)$, it follows that the support of $k_\infty$ is all $C^\infty(S^1, \R^3)$ and in particular every knot (i.e. isotopy class of topological embeddings $S^1\to \R^3$, a set with nonempty interior in the $C^\infty$ topology) has positive probability by Theorem \ref{thm:Kostlan}.3. Moreover, denoting by $\gamma_1\sim \gamma_2$ two isotopic knots, we have that
\be 
\P\left(\de\{k_\infty\sim \gamma\}\right)\le\P\{k_\infty \text{ is not an immersion}\}=0
\ee
by Theorem \ref{thm:Kostlan}.4, because the condition of being an immersion is equivalent to that of being transverse to the zero section of $J^1(S^1,\R^3)\to S^1\times \R^3$. 
Theorem \ref{thm:Kostlan}.2, thus implies that for every knot $\gamma:S^1\to \R^3$ we have:
\be \lim_{d\to \infty}\PP \{k_d\sim \gamma\}=\PP\{k_\infty\sim \gamma\}>0.\ee
 
\end{example}

\bibliographystyle{plain}

\bibliography{Maximal_and_Typical_Topology_of_Real_Polynomial_Singularities}

\begin{thebibliography}{10}

\bibitem{Billingsley}
Patrick Billingsley.
\newblock {\em Convergence of probability measures}.
\newblock Wiley Series in Probability and Statistics: Probability and
  Statistics. John Wiley \& Sons, Inc., New York, second edition, 1999.
\newblock A Wiley-Interscience Publication.

\bibitem{BKL}
Paul Breiding, Hanieh Keneshlou, and Antonio Lerario.
\newblock Quantitative singularity theory for random polynomials, 2019.

\bibitem{buerg:07}
Peter B{\"u}rgisser.
\newblock Average {E}uler characteristic of random real algebraic varieties.
\newblock {\em C. R. Math. Acad. Sci. Paris}, 345(9):507--512, 2007.

\bibitem{CS}
Dustin Cartwright and Bernd Sturmfels.
\newblock The number of eigenvalues of a tensor.
\newblock {\em Linear Algebra Appl.}, 438(2):942--952, 2013.

\bibitem{delbruck}
M.~Delbruck.
\newblock Knotting problems in biology.
\newblock {\em Plant Genome Data and Information Center collection on
  computational molecular biology and genetics}, 1961.

\bibitem{DiattaLerario}
Daouda~Niang Diatta and Antonio Lerario.
\newblock Low degree approximation of random polynomials, 2018.

\bibitem{EdelmanKostlan95}
Alan Edelman and Eric Kostlan.
\newblock How many zeros of a random polynomial are real?
\newblock {\em Bull. Amer. Math. Soc. (N.S.)}, 32(1):1--37, 1995.

\bibitem{eliash}
Y.~Eliashberg, N.M. Mishachev, and S.~Ariki.
\newblock {\em Introduction to the $h$-Principle}.
\newblock Graduate studies in mathematics. American Mathematical Society, 2002.

\bibitem{frischwasserman}
Harry~L. Frisch and Edel Wasserman.
\newblock Chemical topology.
\newblock {\em Journal of the American Chemical Society}, 83 (18):3789–3795,
  1961.

\bibitem{FLL}
Y.~V. Fyodorov, A.~Lerario, and E.~Lundberg.
\newblock On the number of connected components of random algebraic
  hypersurfaces.
\newblock {\em J. Geom. Phys.}, 95:1--20, 2015.

\bibitem{GaWe1}
Damien Gayet and Jean-Yves Welschinger.
\newblock Lower estimates for the expected {B}etti numbers of random real
  hypersurfaces.
\newblock {\em J. Lond. Math. Soc. (2)}, 90(1):105--120, 2014.

\bibitem{GaWe3}
Damien Gayet and Jean-Yves Welschinger.
\newblock Expected topology of random real algebraic submanifolds.
\newblock {\em J. Inst. Math. Jussieu}, 14(4):673--702, 2015.

\bibitem{GaWe2}
Damien Gayet and Jean-Yves Welschinger.
\newblock Betti numbers of random real hypersurfaces and determinants of random
  symmetric matrices.
\newblock {\em J. Eur. Math. Soc. (JEMS)}, 18(4):733--772, 2016.

\bibitem{GoreskyMacPherson}
M.~Goresky and R.~MacPherson.
\newblock {\em Stratified Morse Theory}.
\newblock Ergebnisse der Mathematik und ihrer Grenzgebiete. Springer-Verlag,
  1988.

\bibitem{Hirsch}
Morris~W. Hirsch.
\newblock {\em Differential topology}, volume~33 of {\em Graduate Texts in
  Mathematics}.
\newblock Springer-Verlag, New York, 1994.
\newblock Corrected reprint of the 1976 original.

\bibitem{ItoNisio}
Kiyosi Ito and Makiko Nisio.
\newblock On the convergence of sums of independent banach space valued random
  variables.
\newblock {\em Osaka J. Math.}, 5(1):35--48, 1968.

\bibitem{Ko90}
E.~Kostlan.
\newblock On the distribution of roots of random polynomials.
\newblock In {\em From {T}opology to {C}omputation: {P}roceedings of the
  {S}malefest ({B}erkeley, {CA}, 1990)}, pages 419--431. Springer, New York,
  1993.

\bibitem{khazOFRET}
K.~Kozhasov.
\newblock On fully real eigenconfigurations of tensors.
\newblock {\em SIAM Journal on Applied Algebra and Geometry}, 2(2):339--347,
  2018.

\bibitem{LerarioJEMS}
A.~Lerario.
\newblock Complexity of intersections of real quadrics and topology of
  symmetric determinantal varieties.
\newblock {\em J. Eur. Math. Soc. (JEMS)}, 18(2):353--379, 2016.

\bibitem{DTGRF1}
Antonio {Lerario} and Michele {Stecconi}.
\newblock {Differential Topology of Gaussian Random Fields}.
\newblock {\em arXiv e-prints}, page arXiv:1902.03805v1, Feb 2019.

\bibitem{Milnor}
J.~Milnor.
\newblock On the {B}etti numbers of real varieties.
\newblock {\em Proc. Amer. Math. Soc.}, 15:275--280, 1964.

\bibitem{NazarovSodin2}
F.~Nazarov and M.~Sodin.
\newblock Asymptotic laws for the spatial distribution and the number of
  connected components of zero sets of {G}aussian random functions.
\newblock {\em Zh. Mat. Fiz. Anal. Geom.}, 12(3):205--278, 2016.

\bibitem{Podkorytov}
S.~S. Podkorytov.
\newblock On the {E}uler characteristic of a random algebraic hypersurface.
\newblock {\em Zap. Nauchn. Sem. S.-Peterburg. Otdel. Mat. Inst. Steklov.
  (POMI)}, 252(Geom. i Topol. 3):224--230, 252--253, 1998.

\bibitem{SarnakWigman}
Peter Sarnak and Igor Wigman.
\newblock Topologies of nodal sets of random band limited functions.
\newblock In {\em Advances in the theory of automorphic forms and their
  {$L$}-functions}, volume 664 of {\em Contemp. Math.}, pages 351--365. Amer.
  Math. Soc., Providence, RI, 2016.

\bibitem{shsm}
M.~Shub and S.~Smale.
\newblock Complexity of {B}ezout's theorem. {II}. {V}olumes and probabilities.
\newblock In {\em Computational algebraic geometry ({N}ice, 1992)}, volume 109
  of {\em Progr. Math.}, pages 267--285. Birkh\"auser Boston, Boston, MA, 1993.

\end{thebibliography}

\end{document}